\documentclass{amsart}

\usepackage{amsmath,amssymb}
\usepackage{color}
\usepackage{mathtools}
\usepackage{comment}

\begin{document}

\makeatletter
\@addtoreset{equation}{section}
\def\theequation{\thesection.\arabic{equation}}
\makeatother

\theoremstyle{definition}
\newtheorem{dfn}{Definition}[section]
\newtheorem{thm}[dfn]{Theorem}
\newtheorem*{th*}{Theorem}
\newtheorem{lem}[dfn]{Lemma}
\newtheorem{prop}[dfn]{Proposition}
\newtheorem{rem}[dfn]{Remark}
\newtheorem{cor}[dfn]{Corollary}
\newtheorem*{cor*}{Corollary}
\newtheorem*{prop*}{Proposition}
\newtheorem{quest}[dfn]{Question}
\newtheorem{conj}[dfn]{Conjecture}

\newcommand{\bbA}{\mathbb{A}}
\newcommand{\bbC}{\mathbb{C}}
\newcommand{\bbR}{\mathbb{R}}
\newcommand{\bbQ}{\mathbb{Q}}
\newcommand{\bbZ}{\mathbb{Z}}
\newcommand{\bbH}{\mathbb{H}}

\newcommand{\calA}{\mathcal{A}}
\newcommand{\calC}{\mathcal{C}}
\newcommand{\calD}{\mathcal{D}}
\newcommand{\calE}{\mathcal{E}}
\newcommand{\calI}{\mathcal{I}}
\newcommand{\calN}{\mathcal{N}}
\newcommand{\calO}{\mathcal{O}}
\newcommand{\calP}{\mathcal{P}}
\newcommand{\calS}{\mathcal{S}}
\newcommand{\calU}{\mathcal{U}}
\newcommand{\calV}{\mathcal{V}}
\newcommand{\calZ}{\mathcal{Z}}

\newcommand{\fraka}{\mathfrak{a}}
\newcommand{\frakH}{\mathfrak{H}}
\newcommand{\frakh}{\mathfrak{h}}
\newcommand{\fraki}{\mathfrak{i}}
\newcommand{\frakg}{\mathfrak{g}}
\newcommand{\frakk}{\mathfrak{k}}
\newcommand{\frakp}{\mathfrak{p}}
\newcommand{\frakX}{\mathfrak{X}}

\newcommand{\bfs}{\mathbf{s}}
\newcommand{\bfe}{\mathbf{e}}
\newcommand{\bfa}{\mathbf{a}}

\newcommand{\rank}{\mathrm{rank}}
\newcommand{\corank}{\mathrm{Corank}}
\newcommand{\im}{\mathrm{Im}}
\newcommand{\Hom}{\mathrm{Hom}}

\newcommand{\GL}{\mathrm{GL}}
\newcommand{\Sym}{\mathrm{Sym}}
\newcommand{\Sp}{\mathrm{Sp}}
\newcommand{\Mat}{\mathrm{Mat}}
\newcommand{\Mp}{\mathrm{Mp}}
\newcommand{\SL}{\mathrm{SL}}
\newcommand{\SO}{\mathrm{SO}}

\newcommand{\automforms}{\mathcal{A}(\Gamma)}
\newcommand{\NHMFonG}{\mathcal{N}(\Gamma)}
\newcommand{\NHMFonGwtlam}{\mathcal{N}_\lambda(\Gamma)}
\newcommand{\gk}{(\frakg, K_\infty)}
\newcommand{\NHMFonGcharchi}{\mathcal{N}(\Gamma,\chi)}
\newcommand{\NHAFonPG}{\mathcal{N}(P \backslash G)}
\newcommand{\NHAFonG}{\mathcal{N}(G)}
\newcommand{\NHAFonBG}{\mathcal{N}(B \backslash G)}

\newcommand{\ul}{\underline}
\newcommand{\pr}{\mathrm{pr}}
\newcommand{\Ext}{\mathrm{Ext}}
\newcommand{\fin}{{\rm \mathchar`- fin}}
\newcommand{\fini}{\mathrm{fin}}
\newcommand{\Ind}{\mathrm{Ind}}
\newcommand{\Res}{\mathrm{Res}}
\newcommand{\bs}{\backslash}
\newcommand{\rclub}{{\color{red}\clubsuit}}

\title{Nearly holomorphic automorphic forms on $\mathrm{SL}_2$}
\author{Shuji Horinaga}

\begin{abstract}
	We define the space of nearly holomorphic automorphic forms on a connected reductive group $G$ over $\bbQ$ such that the homogeneous space $G(\bbR)^1/ K_\infty^\circ$ is a Hermitian symmetric space.
	By Pitale, Saha and Schmidt's study, there are the classification of indecomposable $(\frakg,K_\infty)$-modules which occur in the space of nearly holomorphic elliptic modular forms and Siegel modular forms of degree $2$.
	This paper studies global representations of the adele group $G(\bbA_\bbQ)$ which occur in the space of  nearly holomorphic Hilbert modular forms.
	In the case of elliptic modular forms, the result of this paper is an adelization of Pitale, Saha and Schmidt's result.
\end{abstract}

\maketitle

\section{Introduction}
	
	The notion of a nearly holomorphic modular form was introduced by Shimura to study the analytic behavior of Eisenstein series at certain points (cf.~\cite{82_shimura, 86_shimura}).
	For example, the Eisenstein series $E_2$ of weight two is a nearly holomorphic elliptic modular form.
	Similarly, a nearly holomorphic Siegel Eisenstein series of weight $(n+3)/2$ of degree $n$ is constructed by Shimura in \cite[$\S$ 17]{00_shimura}.
	Recently, Pitale, Saha and Schmidt studied the representation theoretic aspects of nearly holomorphic elliptic modular forms in \cite{pss1} and Siegel modular forms of degree $2$ in \cite{pss2} by using the theory of BGG category $\calO$.
	They classified the $(\mathfrak{sl}_2(\bbC),\SO(2,\bbR))$-modules and $(\mathfrak{sp}_4(\bbC),\mathrm{U}(2))$-modules that occur in the space of nearly holomorphic modular forms.
	Moreover they determined the multiplicity of each irreducible $(\mathfrak{sl}_2(\bbC),\SO(2,\bbR))$-module in terms of modular forms.
	In this paper we consider global representations of the adele group generated by nearly holomorphic Hilbert modular forms.
	We determine the structure of the space of non-cuspidal nearly holomorphic automorphic forms on the adele group $\SL_2(\bbA_F)$ as a global representation for a totally real number field $F$.

	Let $G$ be a connected reductive group over $\bbQ$ and $G_\infty$ the group of $\bbR$-valued points of $G$.
	In this paper we assume the homogeneous space $\frakH=G_\infty^\circ / A_G^\infty K_\infty^\circ$ is a Hermitian symmetric space. 
	Here $G_\infty^\circ$ is the identity component of $G(\bbR)$, $K_\infty^\circ$ is a maximal compact subgroup of $G_\infty^\circ$ and $A_G^\infty$ is the identity component of the split component of $G_\infty^\circ$.
	Let $\frakg$, $\frakk$ and $\fraka$ be the complexification of the Lie algebra of $G_\infty^\circ$, $K_\infty^\circ$ and  $A_G^\infty$, respectively.
	Then we have the well-known decomposition:
	\[
	\frakg = \fraka \oplus ( \frakk + \frakp_+ + \frakp_-),
	\]
	where $\frakp_+$ (resp.~$\frakp_-$) is the Lie subalgebra of $\frakg$ corresponding to the holomorphic tangent space (resp.~anti-holomorphic tangent space) of $\frakH$ at the base point $K_\infty^\circ \in \frakH = G_\infty^\circ / A_G^\infty K_\infty^\circ$.
	Here $\oplus$ and $+$ are the direct sum as Lie algebras and the direct sum as $\bbC$-vector spaces, respectively.
	For details, see \cite[\S 2.3, (3.9)]{satake}.
	We denote by $\calA(A \bs G)$ the space of automorphic forms on the adele group $A_G^\infty \bs G(\bbA_\bbQ)$.
	We say that an automorphic form $\varphi$ is nearly holomorphic (resp.~holomorphic) if $\varphi$ is $\frakp_-$-finite (resp.~if $\varphi$ is annihilated by $\frakp_-$), i.e., there exists a positive integer $m \in \bbZ_{\geq 1}$ such that $\frakp_-^m \cdot \varphi=0$ (resp. $\frakp_- \cdot \varphi = 0$).
	Let $\calN(A \bs G)$ be the space of nearly holomorphic automorphic forms, $\calS\calN(A \bs G)$ the space of cusp forms in $\calN(A \bs G)$ and $\calE\calN(A \bs G)$ the orthogonal complement of $\calS\calN(A \bs G)$ in $\calN(A \bs G)$ with respect to the Petersson inner product.
	We would like to understand the space $\calE\calN(A \bs G)$ as a $G(\bbA_\fini) \times (\frakg,K_\infty^\circ)$-module.

	Let $G= \Res_{F/\bbQ} \SL_2$.
	Here $\Res$ is the Weil restriction and $F$ is a totally real number field of degree $d$.
	For a place $v$, we denote by $F_v$ the $v$-completion of $F$.
	Let $\calO_{F_v}$ (resp.~$\calO_F$) be the ring of integers of $F_v$ (resp.~$F$). 
	We denote by $\bfa$ the set of archimedean places of $F$.
	First we consider $(\frakg, K_\infty)$-modules generated by Hilbert modular forms (cf.~\cite{pss1}).
	For a congruence subgroup $\Gamma$ of $\SL_2(F)$, we denote by $\calA(\Gamma)$ the space of automorphic forms on $\Gamma \bs G_\infty$.
	The Lie group $G_\infty = \prod_{v \in \bfa} \SL_2(F_v)$ acts on the Hermitian symmetric space $\frakH=\prod_{v \in \bfa}\frakH_v$.
	In this case, $\frakH_v$ is the complex upper half plane and the maximal compact group $K_v$ of $\SL_2(F_v)$ is the stabilizer of $\sqrt{-1} \in \frakH_v$ under the linear fractional transformation.
	The Lie algebra $\frakk$ is a Cartan subalgebra of $\frakg$ and $\frakp = \frakk + \frakp_-$ is a Borel subalgebra of $\frakg$.
	We choose the positive root system corresponding to the Borel subalgebra $\frakp$.
	For $\lambda \in \bbC^d$, let $N(\lambda)$ be the Verma module with highest weight $\lambda$ and $L(\lambda)$ its irreducible quotient (cf.~\cite[\S 1.3]{cat_o}, \S \ref{BGG_cat}).
	We say that an automorphic form $\varphi \in \calA(\Gamma)$ is nearly holomorphic if $\varphi$ is $\frakp_-$-finite.
	This definition is equivalent to Shimura's definition (cf.~\cite[$\S$ 7]{90_shimura}, \cite[Appendix 8]{00_shimura}).
	Let $\calN(\Gamma)$ be the space of nearly holomorphic automorphic forms in $\calA(\Gamma)$.
	For an infinitesimal character $\chi$, let $\calN(\Gamma,\chi)$ be the generalized $\chi$-eigenspace of $\calN(\Gamma)$.
	We denote by $\chi_\lambda$ the infinitesimal character of $N(\lambda)$ (cf.~\cite{cat_o}).
	For $\lambda=(\lambda_1,\ldots,\lambda_d)$ and $\mu=(\mu_1,\ldots,\mu_d)$, the infinitesimal characters $\chi_\lambda$ and $\chi_\mu$ are the same if and only if there exists $w \in W$ such that $w \cdot \lambda = \mu$.
	Here $W \cong (\bbZ / 2\bbZ)^d$ is the Weyl group of $G_\infty$ and $w \cdot \lambda = w(\lambda + (-1,\ldots,-1)) + (1,\ldots,1)$.
	In this case, for $w=(\varepsilon_1,\ldots,\varepsilon_d) \in W \cong (\bbZ/2\bbZ)^d$ with $\varepsilon_i \in \{\pm1\}$ for any $i$, we have $w(x_1,\ldots,x_d) = (\varepsilon_1 x_1, \ldots, \varepsilon_d x_d)$.
	Then, integral infinitesimal characters are parametrized by the set $\bbZ^d_{\geq 1}$.
	For a parameter $\lambda$, we say that $\lambda$ is singular if there exists $1 \neq w \in W$ such that $w \cdot \lambda = \lambda$.
	We also say that $\lambda$ is regular if $\lambda$ is not singular.
	The following proposition is the Hilbert modular form version of the theorem of \cite{pss1}.
	
	\begin{prop}
	Suppose $F \neq \bbQ$ and $\lambda \in \bbZ^d_{\geq 1}$.
	Then the space $\calN(\Gamma,\chi_\lambda)$ decomposes as the direct sum
	\[
	\calN(\Gamma,\chi_\lambda) = 
	\begin{dcases}
	L(\lambda)^{\bigoplus m(\Gamma,\lambda)}  & \text{if $\lambda \neq (2,\ldots,2)$},\\
	L(\lambda)^{\bigoplus m(\Gamma,\lambda)} \oplus \bbC &\text{if $\lambda=(2,\ldots,2)$}.
	\end{dcases}
	\]
	Here the multiplicity $m(\Gamma,\lambda)$ is the dimension of the space of holomorphic Hilbert modular forms of weight $\lambda$ with respect to $\Gamma$ and $\bbC$ is the trivial representation.
	\end{prop}
	
	Next we ``adelize'' the above statement.
	Set $K= \prod_{v < \infty} K_v \times K_\infty$ where $K_v = \SL_2(\calO_{F_v})$ for a non-archimedean place.
	Let $B$ be the upper triangular subgroup of $\SL_2$ with Levi decomposition $B=MN$.
	Here $M$ is the diagonal subgroup of $\SL_2$.
	For a complex number $s$, a place $v$ of $F$ and a character $\mu_v$ of $F_v^\times$, let $I_v(\mu_v,s)$ be the space of $K_v$-finite functions $f$ on $\SL_2(F_v)$ such that $f(mng) = \mu_v(a) |a|^{s+1} f(g)$ where $m = \mathrm{diag}(a,a^{-1}) \in M(F_v)$, $n \in N(F_v)$ and $g \in G(F_v)$.
	Let $F_{\infty,+}^\times$ be the identity component of the group of units in the ring $F_{\infty} = \prod_{v \in \bfa} F_v$.
	For a Hecke character $\mu$ of $F^\times F_{\infty,+}^\times \bs \bbA_F^\times$, let $I(\mu,s) = \bigotimes_v I_v(\mu_v,s)$.
	Put $P_0=\Res_{F/\bbQ}B$, $M_0 = \Res_{F/\bbQ}M$ and $N_0 = \Res_{F/\bbQ}N$.
	Then $P_0$ is a Borel subgroup of $G$ with Levi decomposition $P_0=M_0N_0$.
	For an automorphic form $\varphi$ on $G(\bbA_\bbQ)$, let $\varphi_0$ be the constant term of $\varphi$ along $P_0$, i.e.,
	\[
	\varphi_0(g) = \int_{N_0(\bbQ) \bs N_0(\bbA_\bbQ)} \varphi(ng)\, dn , \qquad g \in G(\bbA_\bbQ).
	\]
	We denote by $\calS\calN(G)$ the space of automorphic forms $\varphi$ in $\calN(G)$ such that $\varphi_0=0$.
	Let $\calE\calN(G)$ be the orthogonal complement of $\calS\calN(G)$ in $\calN(G)$ with respect to the Petersson inner product and $\calE\calN(G,\chi)$ the generalized eigenspace for an infinitesimal character $\chi$.
	A constant term of nearly holomorphic automorphic form is described as follows:
	
	\begin{prop}
	Let $\varphi$ be a nearly holomorphic automorphic form on $G(\bbA_\bbQ)$.
	Then there exist unitary Hecke characters $\mu_1,\ldots,\mu_\ell$ of $F^\times F_{\infty,+}^\times \bs \bbA^\times_F$, integers $s_1, \ldots, s_\ell$ and functions $ \phi_i \in I(\mu_i,s_i)$ such that the constant term $\varphi_0$ is equal to the sum $\sum_{i=1}^\ell \phi_i$.
	\end{prop}
	
	Hence we can regard the space $\calE\calN(A \bs G)$ as a subrepresentation of the algebraic direct sum of induced representations.
	We consider a necessary condition for $\lambda$ to satisfy $\calE\calN(G,\chi_{\lambda}) \neq 0$.
	We say that a parameter $\lambda = (\lambda_1,\dots,\lambda_d)$ is parallel if $\lambda_1 = \cdots =\lambda_d$ and $\lambda$ is integral if $\lambda \in \bbZ^d$.
	
	\begin{lem}
	For a parameter $\lambda\in\bbC^d$, the space $\calE\calN(G,{\chi_\lambda})$ is zero unless there exists $w \in W$ such that $w \cdot \lambda$ is parallel and integral.
	\end{lem}
	
	Hence we may assume the parameter $\lambda$ lies in $\bbZ^d_{\geq 1}$ and parallel.
	For a parameter $\lambda =(\lambda_1,\ldots,\lambda_d) \in \bbZ^d_{\geq 1}$ and the corresponding infinitesimal character $\chi_\lambda$, we say that $\chi_\lambda$ is sufficiently regular if $\lambda_i \geq 3$ for all $1 \leq i \leq d$.
	For an integer $k$, set $\ul{k} = (k,\ldots,k) \in \bbZ^d$.
	For a place $v$, we denote by $\mathbf{1}_v$ and $\mathbf{1}$ the trivial character of $F_v^\times$ and the trivial Hecke character of $F^\times F_{\infty,+}^\times \bs \bbA^\times_F$, respectively.
	We denote by $\mathfrak{X}$ the set of all Hecke characters of $F^\times F_{\infty,+}^\times \bs \bbA_F^\times$.
	Put
	\[
	\mathfrak{X}_{1} = \{\mu = \otimes_v \mu_v \in \mathfrak{X} \mid \text{$\mu_v = \mathbf{1}_v$ for any $v \in \bfa$}\}
	\]
	and
	\[
	\mathfrak{X}_{-1} = \{\mu = \otimes_v \mu_v \in \mathfrak{X} \mid \text{$\mu_v = \mathrm{sgn}$ for any $v \in \bfa$}\}.
	\]
	Here $\mathrm{sgn}$ is the sign character of $\bbR^\times$.
	
	\begin{thm}[sufficienly regular case]
	For a sufficiently regular infinitesimal character $\chi_{\lambda}$ with $\lambda=\ul{k} \in \bbZ^d_{\geq 3}$, we have
	\[
	\calE\calN(G,\chi_\lambda) \cong \bigoplus_{\mu \in \mathfrak{X}_{(-1)^{k}}} \left(\bigotimes_{v < \infty} I_v(\mu_v,k-1) \otimes  L(\ul{k}) \right).
	\]
	\end{thm}
	
	In order to study the singular infinitesimal character case i.e., $\lambda = \ul{1}$, we need the notion of theta correspondence.
	For a two-dimensional quadratic space $V$ over $F$, let $R(V)$ be the representation of the adele group $G(\bbA)$ corresponding to the trivial representation of the orthogonal group $\mathrm{O}(V)$ under the theta correspondence (see section \ref{ind_rep_theta_corr}).
	Then we may regard the representation $R(V)$ as a subrepresentation of $I(\chi_V,0)$ where $\chi_V$ is the quadratic character associated to $V$ (see \cite[Introduction]{1994_Kudla-Rallis}). 
	For Hecke characters $\mu$ and $\mu'$, we say that $\mu$ is associated to $\mu'$ if $\mu = \mu'$ or $\mu = \mu'^{-1}$.
	If $\mu$ and $\mu'$ are associated, the induced representation $I(\mu,0)$ is isomorphic to $I(\mu',0)$.
	We denote by $\mathfrak{X}_{(-1)^\ell}/ \sim$ is the associated classes of Hecke characters in $\mathfrak{X}_{(-1)^\ell}$.
	
	\begin{thm}[singular case]
	If $\lambda = \ul{1}$, we have
	\[
	\calE\calN(G,\chi_\lambda) \cong \bigoplus_{\begin{smallmatrix}\mu \in \mathfrak{X}_{-1}/\sim \\ \mu^2 \neq \mathbf{1}\end{smallmatrix}} \left(\bigotimes_{v < \infty} I_v(\mu_v,0) \otimes L(\ul{1}) \right) \oplus \bigoplus_V R(V).
	\]
	Here $V$ runs through all isometry classes of two-dimensional quadratic spaces over $F$ such that $V_v$ has signature $(2,0)$ for all $v \in \bfa$.
	\end{thm}
	
	The remaining case is $\lambda = \ul{2}$.
	If $F=\bbQ$, let $E_{2,\bbA}$ be the unique (up to constant) unramified vector of weight $2$ in $\calN(G)$ (see Theorem \ref{NHMF_str}).
	
	\begin{thm}[regular but non-sufficiently regular case]
	Let $\pi$ be the $G(\bbA_\fini) \times (\frakg,K_\infty)$-module generated by $E_{2,\bbA}$.
	\begin{enumerate}
	\item If $F=\bbQ$, we have the following non-split exact sequence
	\[
	0 \longrightarrow \bbC \longrightarrow \pi \longrightarrow \bigotimes_{v < \infty} I_v(\mathbf{1}_v,1) \otimes L(2) \longrightarrow 0.
	\]
	\item If $\lambda = \ul{2}$, we have
	\[
	\calE\calN(G,\chi_\lambda) \cong 
	\begin{dcases}
	\bigoplus_{\mu \in \mathfrak{X}_1} \left(\bigotimes_{v < \infty} I_v(\mu_v,1) \otimes L(\ul{2}) \right) \oplus \bbC & \text{if $F \neq \bbQ$},\\
	\bigoplus_{\begin{smallmatrix}\omega \in \mathfrak{X}_1 \\ \omega \neq \mathbf{1} \end{smallmatrix}} \left(\bigotimes_{v < \infty} I_v(\omega_v,1) \otimes L(2)\right) \oplus \pi & \text{if $F = \bbQ$}.
	\end{dcases}
	\]
	\end{enumerate}
	\end{thm}

	Finally, we compare the $G(\bbA_\fini) \times (\frakg,K_\infty)$-module generated by holomorphic automorphic forms and the space of nearly holomorphic automorphic forms.
	Let $\mathcal{H}(G)$ be the $G(\bbA_\fini)\times (\frakg,K_\infty)$-module generated by holomorphic automorphic forms.
	
	\begin{cor}
	As a $G(\bbA_\fini) \times (\frakg,K_\infty)$-module, we have
	\[
	\calN(G)/\mathcal{H}(G) \cong
	\begin{dcases}
	\mathrm{triv}_{G(\bbA_\fini)} \otimes L(2) &\text{if $F=\bbQ$},\\
	0 &\text{if $F \neq \bbQ$}.
	\end{dcases}
	\]
	Here $\mathrm{triv}_{G(\bbA_\fini)}$ is the trivial representation of $G(\bbA_\fini)$.
 	\end{cor}
	
	Note that the representation $\mathrm{triv}_{G(\bbA_\fini)} \otimes L(2)$ is an irreducible quotient of $\pi$.

	\subsection*{Acknowledgments}
	The auther would like to thank Professor Tamotsu Ikeda for helpful discussions, suggestions and encouragements.
	The auther also would like to thank Ameya Pitale and Ralf Schmidt for helpful discussions.

\section{The space of nearly holomorphic automorphic forms}\label{def_NHMF}
	In this section, we define the space of nearly holomorphic automorphic forms on a connected reductive group $G$ over $\bbQ$ with certain properties.

	\subsection{Definition of nearly holomorphic automorphic forms}
	Let $G$ be a connected reductive group over $\bbQ$.
	In this section, we denote by $\bbA$ the adele ring of $\bbQ$ and by $\bbA_\fini$ the finite adele ring of $\bbQ$.
	For example, $G = \Res_{F/\bbQ} \Sp_{2n}$ where $\Res$ is the Weil restriction and $F$ is a totally real number field.

	For a $\bbQ$-algebra $R$, let $G(R)$ the group of $R$-valued points of $G$.
	For a place $v$ of $\bbQ$ and $R=\bbQ_v$, put $G_v=G(\bbQ_v)$.
	We denote by $\infty$ the archimedean place of $\bbQ$.
	Take a maximal compact subgroup $K_v$ of $G_v$ for any place $v$.
	Set $K = \prod_{v} K_v$.
	We denote by $A_G^\infty$ the identity component of the split component of the center of $G_\infty$.
	
	For any topological group $H$, we denote by $H^\circ$ the identity component of $H$.
	Set
	\[
	G_\infty^{1} = \bigcap_{\chi \in \Hom_{\mathrm{conti}} (G_\infty^\circ,\bbR_+^\times)}^{}\mathrm{Ker}(\chi).
	\]
	It is a normal subgroup which contains the derived subgroup and all compact subgroups of $G_\infty^\circ$.
	Then, by \cite[Lemma 2.2.2]{Wallach}, we have
	\[
	A_G^\infty \cap G_\infty^{1} = \{1\}, \qquad G_\infty^\circ = G_\infty^1 A_G^\infty, \qquad K_\infty^\circ \subset G_\infty^{1}.
	\]
	Hence we have the direct product decomposition $G_\infty^\circ = G_\infty^1 \times A_G^\infty$ as a topological group.
	Since $A_G^\infty$ is connected, the semisimple Lie group $G_\infty^1$ is connected.
	Note that the compact group $K_\infty \cap G_\infty^\circ$ is a connected maximal compact subgroup of $G_\infty^{\circ}$.
	In general, the compact group $K_\infty$ is not connected.
	In this paper, we assume that the homogeneous space $\frakH = G_\infty^1/K_\infty^\circ \cong G_\infty^\circ/A_G^\infty K_\infty^\circ$ is a Hermitian symmetric space (cf.~\cite[$\S$ 2.3]{satake}).
	
	Let $\frakg^1 = \mathrm{Lie}(G_\infty^1) \otimes_\bbR \bbC$ and $\frakk=\mathrm{Lie}(K_\infty^\circ) \otimes_\bbR \bbC$.
	Let $\frakp_+$ (resp.~$\frakp_-$) be the Lie subalgebra of $\frakg^1$ corresponding to the holomorphic tangent space (resp.~anti-holomorphic tangent space) at the base point $K_\infty^\circ \in \frakH = G_\infty^1/K_\infty^\circ$.
	We then have the decomposition
	\[
	\frakg^1 = \frakk + \frakp_+ + \frakp_-.
	\]
	For details, see \cite[$\S$ 2.3, (3.9)]{satake}.
	Put $\frakg = \mathrm{Lie}(G^\circ_\infty) \otimes_{\bbR} \bbC$.
	Let $\calU(\frakg)$ be the universal enveloping algebra of $\frakg$ and $\calZ$ the center of $\calU(\frakg)$.
	
	Fix a minimal parabolic subgroup $P_0$ of $G$ over $\bbQ$.
	We say that a parabolic subgroup $P$ of $G$ is standard if $P$ contains $P_0$.
	We assume that the maximal compact subgroup $K = \prod_v K_v$ of $G(\bbA)$ satisfies the following conditions (cf.~\cite[$\S$ I.1.4]{MW}):
	\begin{itemize} 
	\item $G(\bbA) = P_{0} (\bbA) K$.
	\item $P(\bbA) \cap K = (M(\bbA) \cap K)(N(\bbA) \cap K)$.
	\item $M(\bbA) \cap K$ is a maximal compact subgroup of $M(\bbA)$. 
	\end{itemize}
	Here $P$ runs through all standard parabolic subgroups of $G$ with Levi decomposition $P=MN$ and $M$ is a standard Levi subgroup.

	\begin{dfn}(\cite[Definition I.2.17]{MW}).
	Let $P=MN$ be a standard parabolic subgroup of $G$.
	For a smooth function $\phi: N(\bbA)M(\bbQ) \backslash G(\bbA) \longrightarrow \bbC$, we say that $\phi$ is automorphic if it satisfies the following conditions:
	\begin{itemize}
	\item $\phi$ is right $K$-finite.
	\item $\phi$ is $\mathcal{Z}$-finite.
	\item $\phi$ is slowly increasing.
	\end{itemize}
	We denote by $\mathcal{A}(P \backslash G)$ the space of automorphic forms on $ N(\bbA)M(\bbQ) \bs G(\bbA)$.
	For simplicity, we write $\calA(G)$ when $P=G$.
	Let $\calA(A \bs G)$ be the space of automorphic forms on $A_G^\infty \bs G(\bbA)$.
	The space $\mathcal{A}(P \backslash G)$ is a $G(\bbA_\fini) \times \gk$-module by the right translation.
	For an automorphic form $\varphi$, we say $\varphi$ is nearly holomorphic (resp.~holomorphic) if $\varphi$ is $\frakp_-$-finite (resp.~$\frakp_- \cdot \varphi =0$).
	Let $\NHAFonPG$ be the subspace of $\mathcal{A}(P \backslash G)$ consisting of all nearly holomorphic automorphic forms.
	We call a function $\phi \in \NHAFonPG$ a nearly holomorphic automorphic form.
	Note that the space $\NHAFonPG$ is a $G(\bbA_\fini) \times \gk$-module by the right translation.
	We define $\calA(AP \bs G)$ and $\calN(AP \bs G)$ similarly.
	\end{dfn}
	
	\begin{rem}
	The above definition of near holomorphy seems different from Shimura's definition.
	He defined the notion of near holomorphy of smooth functions on a K\"{a}hler manifold.
	When we lift a function on the Hermitian symmetric spaces to a function on the group (cf.~(\ref{lift_to_group})), our definition and Shimura's definition are equivalent for some cases by \cite[$\S$ 7]{90_shimura}.
	For example, in the case when $G= \mathrm{Res}_{F/\bbQ}\SL_2$ with a totally real number field $F$, the definitions are equivalent (cf.~(\ref{corr_MF_AF})).
	\end{rem}
		
	For an infinitesimal character $\chi : \calZ \longrightarrow \bbC$, let
	\[
	\calN(P \backslash G,\chi) = \{ \phi \in \NHAFonPG \mid z \cdot \phi = \chi(z) \phi, \, \text{for any $z \in \calZ$}\},
	\]
	and $\calN(P \backslash G,\chi)^\text{gen}$ the generalized $\chi$-eigenspace in $\calN(P \bs G)$.
	Since any function in $\NHAFonPG$ is $\calZ$-finite, we have
	\[
	\NHAFonPG = \bigoplus_{\chi} \calN(P\backslash G,\chi)^\text{gen},
	\]
	where $\chi$ runs through all infinitesimal characters.
	
	\begin{quest}
	Does $\NHAFonPG$ decompose as $\bigoplus_{\chi} \calN(P\backslash G,\chi)$ ?
	\end{quest}
	
	In the case where $P=G=\SL_2$ or $\Sp_4$, this question is affirmative by the results of Pitale-Saha-Schmidt \cite{pss1,pss2}.
	Indeed, in such cases, for a nearly holomorphic automorphic form $\varphi$, a $(\frakg,K_\infty)$-module generated by $\varphi$ is semisimple as a $\calZ$-module.

	\subsection{Automorphic forms and classical modular forms}\label{AF_vs_MF}
	We first define a factor of automorphy as follows:
	We follow the book \cite[Chap.~II]{satake}.
	Put $\calP_+=\exp(\frakp_+)$ and $\calP_-=\exp(\frakp_-)$.
	For simplicity, we denote by $K_\bbC$ and $K^\circ_\bbC$ the complexification of $K_\infty$ and the identity component of $K_\bbC$, respectively.
	Note that the connected group $K^\circ_\bbC$ is the complexification of $K_\infty^\circ$.
	We then have
	\[
	\calP_+ A_G^\infty \cap K^\circ_\bbC \calP_- = \{1\}, \qquad G^1_\infty \subset \calP_+ K^\circ_\bbC\calP_-, \qquad G^1_\infty \cap K^\circ_\bbC\calP_- = K_\infty^\circ.
	\]
	Hence we obtain the Harish-Chandra embedding $\frakH \lhook\joinrel\longrightarrow \frakp_+$:
	\[
	\frakH \cong G^\circ_\infty / A_G^\infty K_\infty^\circ \cong G_\infty^1K^\circ_\bbC \calP_-/K^\circ_\bbC \calP_- \lhook\joinrel\longrightarrow \calP_+K^\circ_\bbC\calP_-/K^\circ_\bbC \calP_- \cong \calP_+ \cong \frakp_+.
	\]
	We denote by $\calD$ the image of the embedding.
	For $g \in \calP^+A_G^\infty K_\bbC^\circ \calP_-$, let $(g)_+$, $(g)_{A_G^\infty}$, $(g)_0$ and $(g)_-$ be the components of $g$ corresponding to $\calP_+$, $A_G^\infty$, $K_\bbC^\circ$ and $\calP_-$, respectively,
	i.e.,
	\[
	g= (g)_+(g)_{A_G^\infty}(g)_0(g)_-, \qquad (g)_+ \in \calP_+, \,\, (g)_{A_G^\infty}\in A_G^\infty, \,\, (g)_0 \in K_\bbC^\circ, \,\, (g)_-\in \calP_-.
	\]
	Clearly, the equality $\calD=\{\exp^{-1}((g)_+) \in \calP_+ \mid g \in G_\infty^\circ \}$ holds.
	For $g \in G_\infty^\circ$ and $z \in \calD$, we define elements $g(z) \in \calD$ and $J(g,z) \in K_\bbC^\circ$ by
	\[
	\exp(g(z)) = (g \exp(z))_+, \qquad J(g,z) = (g \exp(z))_0. 
	\]
	The function $J(g,z)$ is called the canonical factor of automorphy.
	For a finite-dimensional representation $(\sigma,V)$ of $K_\bbC^\circ$, put $J_\sigma(g,z) = \sigma (J(g,z))$.
	
	Next we define nearly holomorphic modular forms.
	For $\gamma \in G(\bbQ)$ and a finite-dimensional representation $(\sigma,V)$ of $K_\bbC^\circ$, we define the slash operator $|_\sigma \gamma \colon C^\infty(\frakH,V) \longrightarrow C^\infty(\frakH,V)$ by
	\[
	(f|_\sigma \gamma) (z) = J_\sigma(\gamma,z)^{-1} f(\gamma(z)), \qquad f \in C^\infty(\frakH,V), \,\, z \in \frakH.
	\]  
	For a discrete subgroup $\Gamma$ of $G_\infty^\circ$, we say that a slowly increasing function $f \in C^\infty(\frakH,V)$ is a $C^\infty$-modular form of weight $\sigma$ with respect to $\Gamma$ if $f|_\sigma \gamma = f$ for every $\gamma \in \Gamma$.
	We denote by $C^\infty(\frakH,\sigma)^\Gamma$ the space of all $C^\infty$-modular forms on $\frakH$ of weight $\sigma$ with respect to $\Gamma$.
	Since any Hermitian symmetric space is a K\"{a}hler manifold, we can define the notion of a nearly holomorphic function.
	For details, see \cite{86_shimura,87_shimura,94_shimura,00_shimura}.
	For a discrete subgroup $\Gamma$ of $G_\infty^\circ$ and a finite-dimensional representation $(\sigma,V)$ of $K_\infty^\circ$, we denote by $N_\sigma(\Gamma)$ (resp.~$M_\sigma(\Gamma)$) the subspace of all nearly holomorphic functions (resp.~holomorphic functions) in $C^\infty(\frakH,\sigma)^\Gamma$.

	Let $K_{\text{fin}}$ be an open compact subgroup of $G(\bbA_{\text{fin}})$.
	Let $g_1 , \ldots, g_h$ be a set of complete representatives of the double coset $G(\bbQ) \bs G(\bbA) / G_{\infty}^\circ K_\fini$.
	We may assume that the elements $g_1 , \ldots, g_h$ belong to $G(\bbA_\fini)$, since $G(\bbQ)$ is dense in $G(\bbR)$. 
	Let $\Gamma_i$ be the projection of $G(\bbQ) \cap g_i K_\fini G_\infty^\circ g_i^{-1}$ to $G_\infty^\circ$.
	We now obtain the isomorphism
	\[
	A_G^\infty G(\bbQ) \bs G(\bbA) /  K_\infty K_\fini \cong \coprod_{i=1}^h \Gamma_i \bs \frakH
	\]
	by the map 
	\[
	\gamma g_i g_\infty k \longmapsto g_\infty(\mathbf{i})
	\]
	for any $\gamma \in A_{G}^\infty G(\bbQ), g_\infty \in G_\infty$ and $k \in K_\fini$.
	Here $\mathbf{i}$ is the base point of $\frakH = G_\infty^1/K_\infty^\circ$ corresponding to $K_\infty^\circ$. 
	For a collection $\mathbf{f}=(f_1,\ldots,f_h)$ such that $f_i \in N_\sigma (\Gamma_i)$, we define a function $F_{\mathbf{f}}$ on $G(\bbA)$ by
	\begin{align}\label{lift_to_group}
	F_{\mathbf{f}}(g) = (f_i|_\sigma g_\infty) (\mathbf{i}), \qquad g = \gamma g_i g_\infty k, \,\, \gamma \in G(\bbQ), \,\, g_\infty\in G_\infty^\circ, \, \, k\in K_\fini.
	\end{align}
	Note that the $V$-valued function $F_{\mathbf{f}}$ is left $A_G^\infty G(\bbQ)$-invariant and right $K$-finite.
	Moreover we have
	\begin{align}\label{act_on_F_f}
	F_{\mathbf{f}}(gk_\infty) = \sigma(k_\infty)^{-1} F_{\mathbf{f}}(g)
	\end{align}
	for $g\in G(\bbA), k_\infty\in K_\infty^\circ$.
	Let $(\sigma^\vee,V^\vee)$ be the contragredient representation of $(\sigma,V)$.
	For $v^\vee \in V^\vee$, we have a scalar valued function $g \longmapsto \langle F_{\mathbf{f}}(g), v^\vee \rangle$.
	For simplicity of notation, we denote by $\langle F_{\mathbf{f}}, v^\vee \rangle$ the scalar valued function.
	For an automorphic form $\varphi$ and an finite-dimensional irreducible representation $\tau$ of $K_\infty^\circ$, we say that $\varphi$ has a $K_\infty^\circ$-type $\tau$ if the representation $\langle r(k) \varphi \mid k \in K_\infty^\circ\rangle_\bbC$ of $K_\infty^\circ$ is isomorphic to $\tau$ where $r$ is the right translation.
	Then the automorphic form $\langle F_{\mathbf{f}}, v^\vee \rangle$ has the $K_\infty^\circ$-type $\sigma$ by the formula (\ref{act_on_F_f}).
	By the definition of $F_\mathbf{f}$, the open compact group $K_\fini$ fixes the automorphic form $\langle F_{\mathbf{f}}, v^\vee \rangle$ under the right translation.
	We denote by $\calA(A \bs G)^{K_\fini}_\sigma$ the space of $K_\fini$-fixed functions with $K_\infty^\circ$-type $\sigma$.
	We then obtain the map
	\[
	\left(\bigoplus_{i=1}^h C^\infty(\frakH,\sigma)^{\Gamma_i}\right) \otimes V^\vee \longrightarrow \calA(A \bs G)^{K_\fini}_{\sigma^\vee} \colon \mathbf{f} \otimes v^\vee \longmapsto \langle F_{\mathbf{f}},v^\vee \rangle.
	\]
	
	Conversely, we construct a modular form on $\frakH$ from an automorphic form on $A_G^\infty \bs G(\bbA)$.
	Let $\varphi$ be an automorphic form in $\calA(A \bs G)^{K_\fini}_{\sigma^\vee}$.
	We fix an isomorphism $W = \langle k \cdot \varphi \mid k \in K_\infty^\circ\rangle_\bbC \cong \sigma$ as a $K_\infty^\circ$-representation.
	We then obtain a subspace $\mathrm{Ev}$ of $\Hom_\bbC(W,\bbC)$ spanned by elements $\mathrm{ev}_g$ for every $g \in G(\bbA)$.
	Here $\mathrm{ev}_g (\Phi) = \Phi(g)$ for any $\Phi \in W$.
	For $k \in K_\infty^\circ$ and $\mathrm{ev}_g \in \mathrm{Ev}$, put $k \cdot \mathrm{ev}_g = \mathrm{ev}_{g k^{-1}}$.
	Then the canonical pairing $\langle \, \cdot \, , \, \cdot \, \rangle$ on $\mathrm{Ev} \times W$ is $K_\infty^\circ$-invariant.
	Hence $\mathrm{Ev}$ is isomorphic to $(\sigma^\vee)^\vee \cong \sigma$ as a $K_\infty^\circ$-representation.
	Let $\iota^\vee$ be the fixed isomorphism $W \cong \sigma^\vee$ and $\iota$ an isomorphism from $\mathrm{Ev}$ to $\sigma$ such that $\langle \mathrm{ev}_g,\Phi \rangle = \langle \iota(\mathrm{ev}_g), \iota^\vee(\Phi) \rangle$ for every $\mathrm{ev}_g \in \mathrm{Ev}$ and $\Phi \in W$.
	Let $\{v_j\}$ be a basis of $V$ and $\{v_j^\vee\}$ its dual basis in $V^\vee$.
	We now describe $\mathrm{ev}_g \in \mathrm{Ev}$ in terms of $\varphi$.
	
	\begin{lem}
	With the above notation, we have
	\[
	(\dim V) \sum_{j=1}^{\dim V} \left( \int_{K_\infty^\circ} \varphi(g k^{-1}) \sigma(k^{-1})v_j \, dk \right) \otimes v_j^\vee = \iota(\mathrm{ev}_g) \otimes \iota^\vee(\varphi).
	\]
	\end{lem}
	\begin{proof}
	We have
	\begin{align*}
	\left\langle \int_{K_\infty^\circ} \varphi(gk^{-1}) \sigma(k^{-1}) v_i \, dk, v_j^\vee \right\rangle
	&= \int_{K_\infty^\circ} \varphi(gk^{-1}) \langle\sigma(k^{-1}) v_i, v_j^\vee \rangle\, dk\\
	&= \int_{K_\infty^\circ} \langle k \cdot \mathrm{ev}_g, \varphi\rangle \langle\sigma(k^{-1}) v_i, v_j^\vee \rangle\, dk \\
	&= \int_{K_\infty^\circ} \langle k \cdot \iota(\mathrm{ev}_g), \iota^\vee(\varphi)\rangle \langle\sigma(k^{-1}) v_i, v_j^\vee \rangle\, dk \\
	&= \frac{1}{\dim V} \langle v_i , \iota^\vee(\varphi) \rangle \langle \iota(\mathrm{ev}_g), v_j^\vee \rangle
	\end{align*}
	for any $j$ by the Schur orthogonality.
	Hence we obtain the equality
	\[
	\int_{K_\infty^\circ} \varphi(gk^{-1}) \sigma(k^{-1}) v_i \, dk = \frac{1}{\dim V} \langle v_i, \iota^\vee(\varphi) \rangle \iota(\mathrm{ev}_g).
	\]
	Then the sum
	\[
	(\dim V) \sum_{j=1}^{\dim V} \int_{K_\infty^\circ} \varphi(g k^{-1}) \sigma(k^{-1})v_j \, dk \otimes v_j^\vee
	\]
	is equal to
	\[
	\sum_{j=1}^{\dim V}\langle v_j, \iota^\vee(\varphi) \rangle \iota(\mathrm{ev}_g) \otimes v_j^\vee = \sum_{j=1}^{\dim V} \iota(\mathrm{ev}_g) \otimes \langle v_j, \iota^\vee(\varphi) \rangle v_j^\vee = \iota(\mathrm{ev}_g) \otimes \iota^\vee(\varphi).
	\]
	This completes the proof.
	\end{proof}
	 
	Consequently we obtain $V$-valued functions $f_{i,\varphi}^{(j)}$ on $\frakH$ for $i = 1,\ldots, h$ and $j = 1, \ldots, \dim V$ defined by
	\[
	f_{i,\varphi}^{(j)}(g_\infty(\mathbf{i})) = (\dim V) J_\sigma(g_\infty,\mathbf{i}) \int_{K_\infty^\circ} \varphi(g_i gk_\infty^{-1}) \sigma(k_\infty^{-1}) v_j \, dk_\infty.
	\]
	Then the function $f_{i,\varphi}^{(j)}$ is a $C^\infty$-modular form of weight $\sigma$ with respect to $\Gamma_i$.
	Hence we obtain the converse map
	\[
	\calA(A \bs G)^{K_\fini}_{\sigma^\vee} \xrightarrow{\, \sim \,} \left(\bigoplus_{i=1}^h C^\infty(\frakH,\sigma)^{\Gamma_i}\right) \otimes V^\vee
	\]
	defined by
	\[
	\varphi \longmapsto \sum_{j =1}^{\dim V} \left( f_{i,\varphi}^{(j)} \otimes v_j^\vee \right).
	\]
	To sum it up, we have the following:
	
	\begin{lem}
	Let $(\sigma,V)$ be an irreducible finite-dimensional representation of $K_\infty^\circ$.
	We regard $\sigma$ as the irreducible holomorphic representation of $K_\bbC^\circ$.
	Let $K_\fini$ be an open compact subgroup of $G(\bbA_\fini)$ and $g_1,\ldots,g_h \in G(\bbA_\fini)$ the representatives of $G(\bbQ) \bs G(\bbA) / G_\infty^\circ K_\fini$.
	Let $\Gamma_i$ be the congruence subgroup of $G_\infty^\circ$ corresponding to the double coset $G(\bbQ) g_i G_\infty^\circ K_\fini$.
	With the above notation, we obtain the isomorphism
	\[
	\left(\bigoplus_{i=1}^h C^\infty(\frakH,\sigma)^{\Gamma_i}\right) \otimes V^\vee \cong \calA(A \bs G)^{K_\fini}_{\sigma^\vee} \colon \mathbf{f} \otimes v^\vee \longmapsto \langle F_{\mathbf{f}},v^\vee \rangle.
	\]
	\end{lem}
	
	Under certain assumptions of $G_\infty^\circ$, the above isomorphism induces the following isomorphism:
	
	\begin{align}\label{corr_MF_AF}
	\left(\bigoplus_{i=1}^h N_\sigma(\Gamma_i)\right) \otimes V^\vee \cong \calN(A \bs G)^{K_\fini}_{\sigma^\vee},
	\end{align}
	where $\calN(A \bs G)^{K_\fini}_{\sigma^\vee}$ is the subspace of all $\frakp_-$-finite functions in $\calA(A \bs G)^{K_\fini}_{\sigma^\vee}$.
	For details see section $5$ and $7$ in \cite{90_shimura}.
	In particular if $G=\Res_{F / \bbQ} \Sp_{2n}$ with a totally real number field $F$, the assumptions hold.

	\section{$G = \mathrm{Res}_{F/\bbQ}\SL_2$ case}
	In this section, we determine the structure of $\calE\calN(\Res_{F / \bbQ}\SL_2)$ as a $(\frakg,K_\infty)$-module.
	Throughout this section, let $F$ be a totally real number field over $\bbQ$ of degree $d$ and let $G=\Res_{F / \bbQ}\SL_2$.
	In this case, a maximal compact subgroup $K_\infty$ of $G(\bbR)$ is connected.
	
	\subsection{The category $\calO$}\label{BGG_cat}
	In this subsection, we recall some basic facts of the BGG category $\calO$.
	For details, see \cite{cat_o}.
	
	For simplicity we assume $F = \bbQ$.
	In this case, the Lie algebra $\frakg$ is equal to $\mathfrak{sl}_2(\bbC)$ with Cartan subalgebra $\frakk$.
	Let $\frakp = \frakk + \frakp_-$ be a Borel subalgebra of $\frakg$.
	In the rest of the paper, we choose the positive roots system corresponding to the Borel subalgebra $\frakp$.
	Then the BGG category $\calO$ is defined to be the full subcategory of the category of $\calU(\frakg)$-modules whose objects are the modules $M$ satisfying the following three conditions:
	\begin{itemize}
	\item $M$ is a finitely generated $\calU(\frakg)$-module.
	\item $M$ is $\frakk$-semisimple, that is, $M$ is a weight module: $M = \bigoplus_{\lambda \in \frakk^*} M_\lambda$.
	\item $M$ is locally $\frakp_-$-finite: for each $v \in M$, the subspace $\calU(\frakp_-) \cdot v$ of $M$ is finite-dimensional.
	\end{itemize}
	
	Let $M$ be a $\calU(\frakg)$-module.
	We say that $M$ is a highest weight module if there exists a maximal weight vector $v^+ \in M$ such that $M = \calU(\frakg)\cdot v $ (see \cite[\S 1.2]{cat_o}).
	Any $\lambda \in \frakk^*$ defines $\bbC_\lambda$ a one-dimensional $\frakp$-module $\bbC_\lambda$ with trivial $\frakp_-$-action.
	Now set 
	\[
	N(\lambda) = \calU(\frakg) \otimes_{\calU(\frakp)} \bbC_\lambda
	\]
	which has a natural structure of a left $\calU(\frakg)$-module.
	This is called a Verma module with highest weight $\lambda$.
	Then $N(\lambda)$ is the universal highest weight module of weight $\lambda$.
	The module $N(\lambda)$ has a unique irreducible quotient $L(\lambda)$.
	
	The center $\calZ$ of $\calU(\frakg)$ acts on $N(\lambda)$ as a character since $N(\lambda)$ has a unique highest weight vector.
	We denote by $\chi_\lambda$ the character of $\calZ$.
	Then any character of $\calZ$ is equal to $\chi_\lambda$ for some $\lambda \in \frakk^*$.
	As usual, $\frakk^*$ is identified with $\bbC$ and the unique positive root is identified with the integer $-2$.
	Then for $\lambda$ and $\mu \in \bbC$, the equation $\chi_\lambda = \chi_\mu$ holds if and only if $\lambda = \mu$ or $\lambda = 2 - \mu$.
	For $\lambda \in \bbC$, let $\calO_\lambda$ be the full subcategory of $\calO$ whose objects are all generalized $\chi_\lambda$-eigen modules.
	Then the modules $N(\lambda)$, $N(2-\lambda)$ and their irreducible quotients belong to $\calO_\lambda$.
	For a module $M$, let $M^\vee$ be the contragredient module of $M$.
	Then $N(\lambda)^\vee$ lies in $\calO_\lambda$.
	Since $\calO_\lambda$ has enough projectives, there exists a projective cover $\mathrm{pr} \colon P(\lambda) \longrightarrow L(\lambda)$.
	Here $\mathrm{pr}$ is an essential epimorphism, i.e., no proper submodule of the projective module $P(\lambda)$ is mapped onto $L(\lambda)$.
	Such a module is unique up to isomorphism.
	The classification of indecomposable modules in $\calO_\lambda$ is as follows:

	\begin{lem}[\S 3.12 in \cite{cat_o}]
	With the above notation, we have the following:
	\begin{enumerate}
	\item The Verma module $N(\lambda)$ is irreducible unless $\lambda \in \bbZ_{\leq 0}$.
	\item If $\lambda \not \in \bbZ$, every indecomposable module in $\calO_\lambda$ is isomorphic to $N(\lambda)$ or $N(2-\lambda)$.
	Here the Verma modules $N(\lambda)$ and $N(2-\lambda)$ are irreducible.
	\item If $\lambda = 1$, every indecomposable module in $\calO_\lambda$ isomorphic to $N(1)$.
	Here the Verma module $N(1)$ is irreducible.
	\item If $\lambda \in \bbZ_{> 1}$, every indecomposable module in $\calO_\lambda$ is isomorphic to one of the following five modules:
	\[
	N(\lambda), L(2-\lambda), N(2-\lambda), N(2-\lambda)^\vee, P(\lambda).
	\]
	Moreover we have the following exact sequences:
	\begin{align*}
	0 \longrightarrow N(\lambda) \longrightarrow N(2-\lambda) \longrightarrow L(2-\lambda) \longrightarrow 0,\\
	0 \longrightarrow N(\lambda) \longrightarrow P(\lambda) \longrightarrow N(2-\lambda) \longrightarrow 0.
	\end{align*}
	\end{enumerate}
	\end{lem}

	For $\lambda \in \frakk^*$, we denote by $N^-(\lambda)$ the Verma module with respect to the Borel subalgebra $\frakp_+ + \frakk$.
	Let $L^-(\lambda)$ be its irreducible quotient.
	In this case, $N^-(\lambda)$ is irreducible unless $\lambda \in \bbZ_{\geq 0}$.
	Moreover, with the above notation, $N^-(\lambda)$ has the infinitesimal character $\chi_{-\lambda}$.
	For $\lambda^+ \in \bbZ_{\geq 1}$ and $\lambda^- \in \bbZ_{\leq -1}$, the module $N(\lambda^+)$ (resp.~$N^-(\lambda^-)$) is isomorphic to the (limit of) holomorphic discrete series representation (resp.~(limit of) anti-holomorphic discrete series representation) of $\SL_2(\bbR)$ of weight $\lambda^+$ (resp.~$\lambda^-$).

	For a general totally real field $F$, the Lie algebra $\frakg$ is isomorphic to the $d$-th product $\mathfrak{sl}_2(\bbC) \times \cdots \times \mathfrak{sl}_2(\bbC)$.
	Hence any irreducible module of $\frakg$ is an outer tensor product of certain irreducible modules of $\mathfrak{sl}_2(\bbC)$.
	If $F \neq \bbQ$, any $(\frakg,K_\infty)$-module, which occur in the space of nearly holomorphic automorphic forms on $G(\bbA)$, is semisimple (see Corollary \ref{str_th_inf}). 
	
	\subsection{Induced representations and theta correspondence}\label{ind_rep_theta_corr}
	Let $\bbA_F$ be the adele ring of $F$ and $\bbA_{F,\fini}$ the finite adele ring $F$.
	To simplify notation, we write $\bbA$ and $\bbA_{\fini}$ instead of the adele ring of $\bbQ$ and the finite adele ring $\bbA_{\bbQ,\fini}$, respectively.
	In order to state the main theorem, we recall the theta correspondence and the socle series of induced representations briefly.
	
	For a place $v$, take a two-dimensional quadratic space $(V_v,(\,,\,))$ over $F_v$.
	Let 
	\[
	\Delta(V_v) = - \det(V_v) \in F_v^\times/ F_v^{\times 2}
	\]
	be the discriminant of $V_v$ where $\det(V_v) = \det((x_i,x_j))$ for any basis $\{x_1, x_2\}$ of $V_v$.
	For $x\in F^\times_v$, set
	\[
	\chi_{V_v}(x) = (x,\Delta(V_v))_v,
	\]
	where $(\,\cdot\,,\,\cdot\,)_v$ is the Hilbert symbol of $F_v$.
	We define the Hasse invariant $\varepsilon_v$ of $V_v$ by
	\[
	\varepsilon_v(V_v) = (a_1,a_2)_{v},
	\]
	where  $\{x_1,x_2\}$ is a basis of $V$ such that $(x_i,x_j) = \delta_{i,j}a_i$.
	Note that for a place $v$, the isometry classes of two-dimensional quadratic spaces $V_v$ are determined by the quadratic character $\chi_{V_v}$ and the Hasse invariant $\varepsilon(V_v)$.
	Since the quadratic space $V_v$ is two-dimensional, we have $\varepsilon_v(V_v) = 1$ if $\chi_{V_v}=\mathbf{1}_v$.
	Here $\mathbf{1}_v$ is the trivial character. 	
	
	We take a non-trivial additive character $\psi = \otimes_v \psi_v$ as follows:
	If $F=\bbQ$, let
	\begin{align*}
	\psi_p(x) &= \exp(-2\pi\sqrt{-1}\, y), \qquad x \in \bbQ_p,  \\
	\psi_\infty(x) &= \exp(2\pi\sqrt{-1} \, x),\qquad x \in \bbR, 
	\end{align*}
	where $y \in \cup_{m=1}^\infty p^{-m}\bbZ$ such that $x-y \in \bbZ_p$.
	In general, for an archimedean place $v$ of $F$, put $\psi_v = \psi_\infty$ and for a non-archimedean place $v$ with the rational prime $p$ divisible by $v$, put $\psi_v(x) = \psi_p(\mathrm{Tr}_{F_v/\bbQ_p}(x))$.
	For a place $v$, we have a local Weil representation $\omega_{\psi_v}$ of $\SL_2(F_v) \times \mathrm{O}(V_v)$ on $\mathcal{S}(V_v)$, the space of Schwartz-Bruhat functions on $V_v$ (see \cite{1994_Kudla-Rallis}).
	Let $I_v(\chi_{V_v},s) =  \Ind_{B(F_v)}^{\SL_2(F_v)}(\chi_{V_v} |\cdot|^s)$ where $B=MN$ is the upper triangular subgroup of $\SL_2$.
	Here $M$ is the diagonal subgroup of $\SL_2$.
	We now obtain a $\SL_2(F_v)$-intertwining map 
	\begin{align}\label{theta_corr}
	\mathcal{S}(V_v) \longrightarrow I_v(\chi_{V_v},0)
	\end{align}
	defined by
	\[
	\varphi \longmapsto ( g \longmapsto \omega_v(g)\varphi(0)).
	\]
	Let $R(V_v)$ be the maximal quotient of $\calS(V_v)$ on which $\mathrm{O}(V_v)$ acts trivially, i.e., the space of $\mathrm{O}(V_v)$ co-invariants in $\mathcal{S}(V_v)$.
	Then, under the map (\ref{theta_corr}), the representation $R(V_v)$ may be identified with a subrepresentation of $I_v(\chi_{V_v},0)$.
	If $v$ is an archimedean place and the real quadratic space $V_v$ has signature $(2,0)$, the set of weights in the representation $R(V_v)$ is equal to $\{1+2m \mid m \in \bbZ_{\geq 0} \}$ by \cite[Proposition 2.1]{1990_Kudla-Rallis}.
	Hence, the representation $R(V_v)$ is isomorphic to the limit of holomorphic discrete series representation of $\SL_2(\bbR)$ of weight $1$.
	For an archimedean place $v$ and a real quadratic place $V_v$ with signature $(p,q)$, set $R(p,q) = R(V_v)$.
	
	If $V$ is a global two-dimensional quadratic space, set $R(V)=\bigotimes_v R(V_v)$.
	On the other hand, we choose a collection  $\calC=\{W_v\}_v$ of local two-dimensional quadratic space in such a way that
	\begin{itemize}
	\item $\varepsilon_v=1$ for almost all $v$,
	\item $\chi_{V_v} = \chi_{W_v}$ for all $v$,
	\end{itemize}
	then we may define a global automorphic representation $R(\calC) = \bigotimes_v R(W_v)$.
	If there is no global quadratic space with the collection $\calC$ as its completion, then we call the collection incoherent.
	For a quadratic Hecke character $\chi$, we then have the decomposition
	\[
	I(\chi,0) = \bigoplus_{V} R(V) \oplus \bigoplus_{\calC} R(\calC)
	\]
	where $V$ runs through all two-dimensional quadratic spaces over $F$ with the associated character $\chi_V = \chi$ and $\calC=\{W_v\}_v$ runs through all incoherent families such that the associated character $\otimes_v \chi_{W_v}$ is equal to $\chi$.
	The spaces $R(V)$ and $R(\calC)$ is characterized as eigenspaces of the intertwining operator (cf.~Remark \ref{eigensp_int_op}).
	Kudla and Rallis discuss the realization of the spaces $R(V)$ and $R(\calC)$.
	
	
	\begin{thm}[Theorem 3.1 in \cite{1994_Kudla-Rallis}]\label{KR}
	The following assertions hold.
	\begin{enumerate}
	\item[(1)] If a collection $\calC$ is incoherent, we have
	\[
	\Hom (R(\calC),\calA(G)) = 0.
	\]
	\item[(2)] For a two-dimensional quadratic space $V$ over $F$, we have
	\[
	\Hom (R(V),\calA(G)) \neq 0.
	\]
	Moreover a realization $R(V) \longrightarrow \calA(G)$ can be given by Eisenstein series.
	\end{enumerate}
	\end{thm}
	
	For the Theorem \ref{KR} (2), the realization is given by Lemma \ref{realization}. 
	
	Fix a generator $\varpi_v$ of the maximal ideal of the ring of integers $\calO_{F_v}$ of $F_v$. 
	For a non-archimedean place $v$ of $F$ and a character $\mu_v \colon F_v \longrightarrow \bbC^\times$, we say that $\mu_v$ is normalized if $\mu_v(\varpi_v)=1$.
	The following lemma is well-known:

	\begin{lem}[\cite{1992_Kudla-Rallis} for a non-archimedien case, Theorem 2.4 in \cite{Muic} for an archimedine case]\label{local}
	For a place $v$ of $F$, let $\mu_v$ be a unitary character of $F^\times_v$.
	Suppose $\mu_v$ is normalized if $v$ is non-archimedean and $\mu_v$ is $F_{v,+}^\times$-invariant if $v$ is archimedean.
	Here $F_{v,+}^\times$ is the identity component of $F_v^\times$.
	\begin{enumerate}
	\item If $v$ is an archimedean place, the induced representation $I_v(\mu_v,s)$ has a $\frakp_-$-finite vector if and only if $s \in 2 \bbZ_{\geq0}$ and $\mu_v = \mathrm{sgn}$ or $s \in -1+2\bbZ_{\geq 0}$ and $\mu_v = \mathbf{1}_v$.
	Moreover, we have
	\[
	0 \longrightarrow L(k) \oplus L^-(k) \longrightarrow I_v(\mathrm{sgn}^k,k-1) \longrightarrow F_k \longrightarrow 0, \qquad I_v(\mathrm{sgn},0) \cong R(2,0) \oplus R(0,2),
	\]
	\[
	0 \longrightarrow \bbC \longrightarrow I_v(\mathbf{1}_v,-1) \longrightarrow L(2) \oplus L^-(2) \longrightarrow 0,
	\]
	for $k \in \bbZ_{\geq1}$.
	Here, the representation $V_k$ is the finite dimensional representation of $\SL_2(\bbR)$ with $\dim F_k = k-1$ and $\bbC$ is the trivial representation.
	
	\item If $v$ is a non-archimedean place, the induced representation $I_v(\mu_v,s)$ is reducible if and only if either of the following conditions holds:
	\begin{itemize}
	\item[$\bullet$] $\mu_v^2=\mathbf{1}_v$, $\mu_v \neq 1$ and $s \in (\pi \sqrt{-1}/ \log q) \bbZ$.
	\item[$\bullet$] $\mu_v =\mathbf{1}_v$ and $s \in \{ \pm 1 + 2m\pi \sqrt{-1} / \log q  \mid m \in \bbZ\}\cup \{(2m+1)\pi \sqrt{-1} / \log q \mid m \in \bbZ\}$.
	\end{itemize} 
	Moreover, if $\mu_v^2=\mathbf{1}_v$ and $\mu_v \neq \mathbf{1}_v$, we have
	\[
	I_v(\mu_v,0) = R(V_v^1) \oplus R(V_v^{-1}),
	\]
	where $V_v^{\pm1}$ is a two-dimensional quadratic space over $F_v$ with $\chi_{V_v^{\pm1}} = \mu_v$ and $\varepsilon(V_v^a)=a$ for $a \in \{\pm1\}$.
	If $\mu_v =\mathbf{1}_v$, we have
	\[
	0 \longrightarrow \mathrm{St}_v \longrightarrow I_v(\mathbf{1}_v,1) \longrightarrow \bbC \longrightarrow 0,
	\]
	\[
	0 \longrightarrow \bbC \longrightarrow I_v(\mathbf{1}_v,-1) \longrightarrow \mathrm{St}_v \longrightarrow 0,
	\]
	where $\mathrm{St}_v$ is the Steinberg representation of $\SL_2(F_v)$.
	\end{enumerate}
	\end{lem}

	\subsection{The space of nearly holomorphic Hilbert modular forms}
	In this subsection we discuss basic properties of nearly holomorphic Hilbert modular forms.
	
	For a finite-dimensional representation $\sigma$ of $K_\infty$ and a congruence subgroup $\Gamma$ of $\SL_2(F)$, let $N_\sigma(\Gamma)$ (resp.~$M_\sigma(\Gamma)$) be the space of nearly holomorphic Hilbert modular forms (resp.~holomorphic Hilbert modular forms) of weight $\sigma$ (cf.~section \ref{def_NHMF}).
	Since the compact connected Lie group $K_\infty$ is isomorphic to $\SO(2,\bbR)^d$, irreducible representations of $K_\infty$ are parametrized by $d$-tuple integers $\mathbf{k} = (k_1, \ldots, k_d) \in \bbZ^d$.
	If the representation $\sigma$ corresponds to the integers $\mathbf{k} \in \bbZ^d$, let us denote $N_\mathbf{k}(\Gamma) =  N_\sigma(\Gamma)$ and $M_\mathbf{k}(\Gamma) = M_\sigma(\Gamma)$.
	For a holomorphic Hilbert modular form $f$ of weight $\mathbf{k} = (k_1,\ldots,k_d)$, we have $k_1 = \dots = k_d$ unless $f$ is a cusp form by \cite[Chap.~1, Remark 4.8]{Freitag}.
	
	For an integer $k$, we define differential operators $R_k$ and $L_k$ on $\frakH_1=\{z \in \bbC \mid \mathrm{Im}(z)>0\}$ by
	\[
	R_k= \frac{k}{y} + 2 \sqrt{-1}\,\frac{\partial}{\partial z} , \qquad  L_k=-2 \sqrt{-1}\, y^2\frac{\partial}{\partial \overline{z}}.
	\]
	Let $\bfa = \{\infty_1,\ldots,\infty_d\}$ be the set of archimedean places of $F$.
	We regard $\bfa$ as the set of embeddings of $F$ into $\bbR$.
	We have $G(\bbR) = \prod_{v \in \bfa} \SL_2(F_v)$.
	We denote by $\frakH_v $ the homogeneous space $\SL_2(F_v)/K_v$.
	The Hermitian symmetric space $\frakH_v$ is the complex upper half plane $\frakH_1$ and the maximal compact subgroup $K_v$ is the stabilizer of $\sqrt{-1} \in \frakH_v$ under the linear fractional transformation.
	We define differential operators $R_{k,j}$ and $L_{k,j}$ on $\frakH_{\infty_j}$ similarly to the above differential operators $R_k$ and $L_k$.
	For $1 \leq i \leq d$, set $e_i = (0, \ldots, 0, 1, 0, \ldots, 0)$.
	Then the spaces $R_{k,j}(N_{\mathbf{k}}(\Gamma))$ and $L_{k,j}(N_{\mathbf{k}}(\Gamma))$ are contained in $N_{\mathbf{k}+2e_j}(\Gamma)$ and $N_{\mathbf{k}-2e_j}(\Gamma)$, respectively.
	For a non-negative integer $\ell$, set 
	\[
	R_{k,j}^{(\ell)} = R_{k+2\ell -2,j} \circ \cdots \circ R_{k+2,j} \circ R_{k,j}, \qquad L_{k,j}^{(\ell)}= L_{k-2\ell +2,j} \circ \cdots \circ L_{k-2,j} \circ L_{k,j}.
	\]
	We say that a differential operator $D$ on $\frakH = \prod_{v \in \bfa} \frakH_v$ is a Maass-Shimura differential operator if $D$ is a composition of the above operator $R_{k,j}^{(\ell)}$ for any $j$.
	When $F=\bbQ$, let $E_2$ be the weight two Eisenstein series with the Fourier expansion
	\begin{align}\label{E_2}
	E_2(x+\sqrt{-1}\,y)=\frac{3}{\pi y} - 1 + 24 \sum_{n=1}^\infty \left(\sum_{0 < d | n}d \right) \exp(2\pi \sqrt{-1} \, n(x+\sqrt{-1}\,y)), \qquad x+\sqrt{-1}\, y \in \frakH_1.
	\end{align}
	The following statement has been proved by Shimura \cite[Theorem 5.2]{87_shimura}:
	
	\begin{thm}\label{NHMF_str}
	Take a nearly holomorphic Hilbert modular form $f \in N_{\mathbf{k}}(\Gamma)$.
	Then there exist holomorphic Hilbert modular forms $f_1, \ldots,f_\ell$ and Maass-Shimura differential operators $D_0,\ldots D_\ell$ such that
	\[
	f = 
	\begin{dcases}
	\sum_{i=1}^\ell D_i f_i + D_0E_2 & \text{if $F=\bbQ$}, \\
	\sum_{i=1}^\ell D_i f_i & \text{if $F \neq \bbQ$}.
	\end{dcases}
	\]
	\end{thm}

	If $F=\bbQ$, the following discussions can be found in \cite{pss1}.
	For the rest of this subsection, we assume $F \neq \bbQ$.
	For a congruence subgroup $\Gamma$ and an infinitesimal character $\chi_\lambda$ (see \S \ref{BGG_cat}), let $\calN(\Gamma,\chi_\lambda)^{\mathrm{gen}}$ be the generalized $\chi_\lambda$-eigenspace of nearly holomorphic automorphic forms on $\Gamma \bs G(\bbR)$ and $\calN(\Gamma,\chi_\lambda)$ the subspace of all $\chi_\lambda$-eigenforms in $\calN(\Gamma,\chi_\lambda)^{\mathrm{gen}}$.
	By $K_\infty$-finiteness of automorphic form, a weight in $\calN(\Gamma,\chi_\lambda)^{\mathrm{gen}}$ is integral.
	Since any module in $\calO$ with integral highest weight has integral infinitesimal character, the space $\calN(\Gamma,\chi_\lambda)^{\mathrm{gen}}$ is zero unless the infinitesimal character $\chi_\lambda$ is integral, i.e., $\lambda \in \bbZ^n$.
	For $\lambda = (\lambda_1, \dots, \lambda_d) \in \bbZ^d$, we denote by $L(\lambda)$ the irreducible highest weight module $L(\lambda_1) \boxtimes \cdots \boxtimes L(\lambda_d)$ of $\frakg \cong \mathfrak{sl}_2(F_v) \otimes_\bbR \bbC$.
	For an integral infinitesimal character $\chi_\lambda$, we may assume $\lambda \in \bbZ^d_{\geq 1}$.

	\begin{cor}\label{str_th_inf}
	Suppose $\lambda \in \bbZ_{\geq 1}^d$ and $F \neq \bbQ$.
	As a $(\frakg,K_\infty)$-module, we have
	\[
	\calN(\Gamma,\chi_\lambda)^{\mathrm{gen}} = \calN(\Gamma,\chi_\lambda) \cong 
	\begin{dcases}
	L(\lambda)^{\bigoplus m(\Gamma,\lambda)} &\text{if $\lambda \neq (2,\ldots,2)$}, \\
	L(\lambda)^{\bigoplus m(\Gamma,\lambda)} \oplus \bbC &\text{if $\lambda=(2,\ldots,2)$}.
	\end{dcases}
	\]
	Here the multiplicity $m(\Gamma,\lambda)$ is the dimension of $M_\lambda(\Gamma)$.
	\end{cor}
	\begin{proof}
	By Theorem \ref{NHMF_str} and $F \neq \bbQ$, any automorphic form in $\calN(\Gamma)$ generates a highest weight representation.
	Since any highest weight module has some infinitesimal character, we have $\calN(\Gamma,\chi_\lambda)^{\mathrm{gen}} = \calN(\Gamma,\chi_\lambda)$.
	Take a nearly holomorphic automorphic form $\varphi \in \calN(\Gamma,\chi_\lambda)$ of weight $\mathbf{k} = (k_1, \ldots, k_d)$.
	We may assume $\varphi$ is highest weight and $k_i \geq 0$ for any $i$.
	If $k_i \geq 1$ for all $i$, the automorphic form $\varphi$ generates the irreducible module $N(k_1,\ldots,k_d)$.
	In this case $\mathbf{k}$ is equal to $\lambda$.
	If $k_i = 0$ for some $i$, we may assume $k_1 = \cdots = k_d=0$.
	Indeed, if $k_j \neq 0$ for some $j \neq i$, we have $M_\mathbf{k}(\Gamma) = 0$ for any congruence subgroup $\Gamma$ by Freitag \cite[Chap.~1 Prop.~4.11]{Freitag}.
	Then the automorphic form $\varphi$ is a constant function.
	In this case one has $\lambda = (2, \ldots, 2)$ by easy computation.
	This completes the proof.
	\end{proof}

	\subsection{Main theorem}
	Let $P_0 = \Res_{F / \bbQ} B$, $M_0=\Res_{F/\bbQ} M$ and $N_0= \Res_{F/\bbQ} N$.
	Then $P_0$ is a Borel subgroup of $G$ with Levi decomposition $P_0=M_0N_0$.
	For $\phi \in \calA(G)$, we define the constant term $\phi_{0} \in \calA(P_0 \bs G)$ by
	\[
	\phi_{0}(g) = \int_{N_0(\bbQ) \bs N_0(\bbA_\bbQ)} \phi(ng) \, dn, \qquad g \in G(\bbA_\bbQ).
	\]

	By the definition of $\calE\calN(G)$, the map
	\[
	\calE\calN(G) \longrightarrow \calA(P_0 \backslash G) \colon \phi \longmapsto \phi_0
	\]
	is injective.
	Note that this map is $G(\bbA_\fini) \times \gk$-intertwining.
	Hence, the image is contained in $\calN(P_0 \bs G)$.
	For an infinitesimal character $\chi:\calZ \longrightarrow \bbC$, let
	\[
	\calE\calN(G,\chi) =\calE\calN(G) \cap \calN(G,\chi).
	\]
	By Corollary \ref{str_th_inf}, the action of $\calZ$ on $\calN(\Gamma,\chi_\lambda)$ is the character $\chi_\lambda$.
	Hence the action of $\calZ$ on $\calE\calN(G)$ is semisimple.
	This implies that
	\[
	\calE\calN(G) = \bigoplus_{\chi} \calE\calN(G,\chi).
	\]

	We say that $\lambda = (\lambda_1 ,\ldots ,\lambda_d) \in \bbC^d$ is parallel if $\lambda_1 = \cdots = \lambda_d$ and is integral if $\lambda \in \bbZ^d$.
	
	\begin{lem}\label{int_par}
	Suppose $\calE\calN(G,\chi_\lambda)$ is non-zero.
	Then there exists $w \in W$ such that $w \cdot \lambda$ is integral and parallel.
	\end{lem}
	
	\begin{proof}
	Integrality of $\chi_\lambda$ is clear by Corollary \ref{str_th_inf}.
	Let $\Pi$ be a non-zero $G(\bbA_\fini)\times(\frakg,K_\infty)$-submodule of $\calE\calN(G,\chi_\lambda)$.
	Take a non-zero automorphic form $\varphi \in \Pi$.
	Let $\Pi_\infty$ be a $(\frakg,K_\infty)$-module generated by $\varphi$.
	Since $\Pi_\infty$ is a locally $\frakp_-$-finite module with finite length, there exists a highest weight vector $\phi$ of weight $\mu$ in $\Pi_\infty$.
	Note that $\phi$ is orthogonal to all cusp forms.
	Since holomorphic non-cusp forms with some weight have parallel weight by \cite[Chap.~1, Remark 4.8]{Freitag}, the weight $\mu$ is parallel and integral.
	By $\chi_\lambda=\chi_\mu$, there exists $w \in W$ such that $\mu = w \cdot \lambda$.
	This completes the proof.
	\end{proof}
	
	Now, we can obtain a description of a constant term.
	
	\begin{prop}\label{const_term_NHAF}
	For a nearly holomorphic automorphic form $\varphi \in \calN(G)$, there exist Hecke characters $\mu_1, \dots, \mu_\ell$ of $F^\times F_{\infty,+}^\times \bs \bbA^\times_F$, integers $s_1, \ldots s_\ell$ and functions $\phi_i \in \Ind_{B(\bbA_F)}^{\SL_2(\bbA_F)}(\mu_i |\cdot|^{s_i})$ such that
	\[
	\varphi_0 = \sum_{i=1}^\ell \phi_{i}.
	\]
	\end{prop}
	
	\begin{proof}
	By the proof of \cite[Lemma I.3.2]{MW} and commutativity of $M_0(\bbA_\bbQ)$, any automorphic form $\phi \in \calA(B \backslash G)$ is left $M_0 (\bbA_\bbQ)$-finite.
	Take $\varphi \in \calN(G)$.
	Since $\varphi$ is $K_\infty$-finite, it decomposes as a finite sum
	\[
	\varphi=\sum_\mathbf{k} \varphi_\mathbf{k},
	\]
	where $\varphi_\mathbf{k}$ is of weight $\mathbf{k}$.
	Then, $\varphi_\mathbf{k}$ is also $\frakp_-$-finite by $\mathrm{Ad}(K_\infty) (\frakp_-) = \frakp_-$.
	We may assume $\varphi=\varphi_{\mathbf{k}}$.
	Define a nearly holomorphic Hilbert modular form $F_\varphi|_{\mathbf{k}} \gamma$ by
	\[
	(F_\varphi |_\mathbf{k} \gamma)(g_\infty(\mathbf{i})) = j(\gamma g_\infty,\mathbf{i})^\mathbf{k} \varphi(g_\infty \gamma_{\fini}^{-1}), \qquad g_\infty \in G(\bbR), \,\, \gamma \in \SL_2(\calO_F),
	\]
	where $\gamma_\fini$ is an element of $G(\bbA_\bbQ)$ defined by
	\[
	(\gamma_\fini)_v = \gamma \quad \text{for $v < \infty$}, \qquad \gamma_\infty = 1 \quad \text{for $v \in \bfa$},
	\]
	and $j(\gamma g_\infty,\mathbf{i})^\mathbf{k} = J_{\sigma_{\mathbf{k}}}(\gamma g_\infty, \mathbf{i})$ (see section \ref{AF_vs_MF}).
	Here $\sigma_\mathbf{k}$ corresponds to the irreducible representation of $K_\infty$ with the parameter $\mathbf{k}$.
	For $\xi \in F$ and $z = (z_1, \ldots,z_d) \in \bbC^d$, set
	\[
	\bfe_\xi(z) = \exp \left(2\pi \sqrt{-1}\, \sum_{j=1}^d \xi^{(j)}z_{j} \right),
	\]
	where $\xi^{(j)} = \infty_j(\xi)$ for $\infty_j \in \bfa$.
	We denote the Fourier expansion of $F_\varphi|_{\mathbf{k}}\gamma$ by
	\[
	(F_\varphi |_\mathbf{k} \gamma)(z) = \sum_{\xi \in F} a(z,\xi,\gamma) \bfe_\xi(z), \qquad z=(z_1, \ldots, z_d) \in \prod_{v \in \bfa}\frakH_v.
	\]
	Since $\varphi$ is $\frakp_-$-finite, the Fourier coefficient $a(z,\xi,\gamma)$ is a polynomial in $y_1^{-1}, \ldots, y_d^{-1}$, where $(y_1,\ldots,y_d)=\mathrm{Im}(z)$.
	For $k \in G(\bbA_\fini)$, let $(\varphi_0)_k$ be the automorphic form on $M_0(\bbA)$ defined by
	\[
	(\varphi_0)_k(m) = \varphi_0(mk).
	\]
	Let $K_\varphi$ be the stabilizer of $\varphi$ in $G(\bbA_\fini)$.
	By the strong approximation theorem, for $k \in G(\bbA_\fini)$, there exists $\gamma_k \in G(\bbQ)$ such that the finite component of $\gamma_k^{-1} k$ lies in the open compact subgroup $K_\varphi$.
	Then, for $g_\infty \in G(\bbR)$ and $k \in G(\bbA_\fini)$, we have
	\[
	\varphi(g_\infty k) = \varphi(\gamma_k^{-1} \gamma_{k,\infty} g_\infty) = (F_\varphi |_{\mathbf{k}} \gamma_{k,\infty} g_\infty)(\mathbf{i}) = j(g_\infty,\mathbf{i})^{-\mathbf{k}}(F_\varphi|_\mathbf{k} \gamma_{k,\infty})(g_\infty(\mathbf{i})).
	\]
	Here $\gamma_{k,\infty}$ is the infinity component of $\gamma_k$.
	For $m_\infty \in M_0(\bbR)$, we then have
	\begin{align*}
	(\varphi_0)_k(m_\infty) 
	&= \int_{N_0(\bbQ) \bs N_0(\bbA)} \varphi(nm_\infty k) \, dn = \int_{L \bs N_0(\bbR)} \varphi(n_\infty m_\infty k)\, dn_\infty \\
	&= \int_{L \bs N_0(\bbR)} \sum_{\xi \in F} j(n_\infty m_\infty,\mathbf{i})^{-\mathbf{k}} a(n_\infty m_\infty(\mathbf{i}),\xi,\gamma_{k,\infty}) \bfe_\xi(n_\infty m_\infty(\mathbf{i})) \, dn_\infty \\
	&= j(m_\infty,\mathbf{i})^{-\mathbf{k}} a(m_\infty(\mathbf{i}),0,\gamma_{k,\infty}).
	\end{align*}
	Here, $L$ is the lattice of $N_0(\bbR)$ corresponding to $N_0(\bbA_\fini) \cap K_\varphi $, i.e., the infinity component of $N_0(\bbQ) \cap (K_\varphi \times N_0(\bbR))$.
	Note that $j(n_\infty m_\infty,\mathbf{i})^\mathbf{k}$ and $a(n_\infty m_\infty(\mathbf{i}),\xi,\gamma_{k,\infty})$ do not depend on $n_\infty$ by the definitions.
	This shows that the action of $M_0(\bbR)$ on $\varphi_0$ under the left translation is semisimple.
	By \cite[Lemma I.3.2]{MW}, there exist Hecke characters $\mu_j$ of $F^\times F_{\infty,+}^\times \bs \bbA_F^\times$ and integers $s_j$ for $1 \leq j\leq \ell$ such that
	\[
	\varphi_0 \in \sum_{j=1}^\ell \Ind_{B(\bbA_F)}^{\SL_2(\bbA_F)}(\mu_j |\cdot|^{s_j}).
	\]
	This completes the proof.
	\end{proof}
	
	We now state the main theorem.
	By Proposition \ref{const_term_NHAF}, the module $\calE\calN(G,\chi)$ is contained in the algebraic direct sum of certain induced representations.
	For $k \in \bbZ$, let
	\[
	\ul{k} = (k,\ldots,k) \in \bbZ^d.
	\]
	For a Hecke character $\mu$ of $F^\times F_{\infty,+}^\times \bs \bbA_F^\times$, let $I(\mu,s)$ be the induced representation $\Ind_{B(\bbA_F)}^{\SL_2(\bbA_F)} (\mu|\cdot|^s)$ and $I_\fini(\mu,s)$ the finite part of the induced representation $I(\mu,s)$.
	We denote by $\mathfrak{X}$ the set of all Hecke characters of $F^\times F_{\infty,+}^\times \bs \bbA_F^\times$.
	Put
	\[
	\mathfrak{X}_{1} = \{\mu = \otimes_v \mu_v \in \mathfrak{X} \mid \text{$\mu_v = \mathbf{1}_v$ for any $v \in \bfa$}\}
	\]
	and
	\[
	\mathfrak{X}_{-1} = \{\mu = \otimes_v \mu_v \in \mathfrak{X} \mid \text{$\mu_v = \mathrm{sgn}$ for any $v \in \bfa$}\}.
	\]
	For Hecke characters $\mu_1$ and $\mu_2$, we say that $\mu_1$ and $\mu_2$ are associate if $\mu_1 = \mu_2$ or $\mu_1 = \mu_2^{-1}$.
	If $\mu_1$ and $\mu_2$ are associate, we write $\mu_1 \sim \mu_2$.
	
	By Theorem \ref{NHMF_str}, if $F=\bbQ$, there exists a unique (up to constant) nearly holomorphic automorphic form of weight two which is invariant under the right translation of $\prod_{v < \infty} \SL_2(\calO_{F_v})$.
	We write $E_{2,\bbA}$ such the automorphic form.
	Under the isomorphism (\ref{corr_MF_AF}), $E_{2,\bbA}$ corresponds to the weight two Eisenstein series $E_2$ (see (\ref{E_2})).

	\begin{thm}\label{main}
	Suppose $\lambda \in \bbZ^d_{\geq 1}$.
	\begin{enumerate}
	\item The space $\calE\calN(G,\chi_\lambda)$ is zero unless $\lambda$ is parallel.
	\item If $k \geq 3$ and $\lambda=\ul{k}$, we have
	\[
	\calE\calN(G,\chi_\lambda) \cong \bigoplus_{\mu \in \mathfrak{X}_{(-1)^{k}}} \left( I_\fini(\mu,k-1) \otimes L(\ul{k}) \right).
	\]
	\item If $k=1$ and $\lambda=\ul{k}$, we have
	\[
	\calE\calN(G,\chi_\lambda) \cong \bigoplus_{\begin{smallmatrix}\mu \in \mathfrak{X}_{-1} / \sim \\ \mu^2 \neq \mathbf{1}\end{smallmatrix}}\left( I_\fini(\mu,0) \otimes L(\ul{1}) \right) \oplus \bigoplus_{V} R(V),
	\]
	where $V$ runs through all isometry classes of two-dimensional quadratic spaces over $F$ such that $V_v$ has signature $(2,0)$ for all $v \in \bfa$.
	\item If $k=2$ and $\lambda = \ul{k}$, we have
	\[
	\calE\calN(G,\chi_\lambda) \cong 
	\begin{dcases}
	\bigoplus_{\begin{smallmatrix}\mu \in \mathfrak{X}_{1} \\ \mu \neq \mathbf{1} \end{smallmatrix}} \left( I_\fini(\mu,1) \otimes L(2)\right) \oplus \pi & \text{if $F=\bbQ$}\\
	\bigoplus_{\omega \in \mathfrak{X}_{1}} \left( I_\fini(\omega,1) \otimes L(\ul{2}) \right) \oplus \bbC &\text{if $F \neq \bbQ$}.
	\end{dcases}
	\]
	Here $\pi$ is a $G(\bbA_\fini) \times (\frakg,K_\infty)$-module generated by $E_{2,\bbA}$.
	Moreover we have a long exact sequence
	\[
	0 \longrightarrow \bbC \longrightarrow \pi \longrightarrow  I_\fini(\mathbf{1},1) \otimes L(2) \longrightarrow 0.
	\]
	\end{enumerate}
	\end{thm}
	
	We compare the space of nearly holomorphic automorphic forms and the space of Eisenstein series (cf.~Appendix \ref{sp_Eis_ser}).
	Conjecture \ref{conj} follows immediately from the proof of Theorem \ref{main}.
	
	\begin{cor}\label{cor_main_1}
	Conjecture \ref{conj} is true for $G = \Res_{F / \bbQ} \SL_2$.
	\end{cor}
	
	Finally, we compare the space generated by holomorphic automorphic forms and the space of nearly holomorphic automorphic forms.
	Let $\mathcal{H}(G)$ be the $G(\bbA_\fini) \times (\frakg,K_\infty)$-module generated by all holomorphic automorphic forms.
	
	\begin{cor}\label{cor_main_2}
	As a $G(\bbA_\fini) \times (\frakg,K_\infty)$-module, we have
	\[
	\calN(G)/\mathcal{H}(G) \cong
	\begin{dcases}
	\mathrm{triv}_{G(\bbA_\fini)} \otimes L(2) &\text{if $F=\bbQ$}\\
	0 &\text{if $F \neq \bbQ$}.
	\end{dcases}
	\]
	Here $\mathrm{triv}_{G(\bbA_\fini)}$ is the trivial representation of $G(\bbA_\fini)$.
	\end{cor}
	
	In the rest of the paper, we prove the above statements.
	
	\subsection{Eisenstein series}
	We review briefly the theory of Eisenstein series on $\SL_2(\bbA_F)$.
	
	Fix a Hecke character $\mu$ of $F^\times F_{\infty,+}^\times \bs \bbA_F$.
	Let $K$ be a maximal compact subgroup of $G(\bbA)$ defined by
	\[
	K_v=\SL_2(\calO_{F_v}), \qquad K=\prod_{v<\infty}^{} K_v \times K_\infty.
	\]
	We say that a section $\Phi_s$ of $I(\mu,s)$ is a standard section if the restriction $\Phi_s|_K$ is independent of $s$.
	For a standard section $\Phi_s$, we define the Eisenstein series by
	\[
	E(g,s,\Phi) = \sum_{\gamma \in P_0(\bbQ) \bs G(\bbQ)} \Phi_s(\gamma g), \qquad g \in G(\bbA).
	\]
	By Langlands \cite{Langlands}, $E$ converges absolutely for $\mathrm{Re}(s) \gg 0$ and can be meromorphically continued to whole $s$-plane (see section \ref{sp_Eis_ser}).
	Moreover, the functional equation
	\[
	E(g,s,\Phi) = E(g,-s,M_s\Phi)
	\]
	holds.
	Here, for $\mathrm{Re}(s)\gg 0$, the intertwining operator $M_s$ is defined by
	\[
	M_s f(g) = \int_{N(\bbA_\bbQ)} f(wng) \, dn , \quad w=\begin{pmatrix} 0 & -1 \\ 1 & 0 \end{pmatrix}.
	\]
	Then, the constant term $E_0$ of $E$ is equal to
	\[
	E_0(g,s,\Phi_s) = \Phi_s(g) + M_s\Phi_s(g).
	\]
	By Langlands \cite{Langlands}, $E_0$ has a pole at $s=s_0$ if and only if $E$ has a pole at $s=s_0$.
	Moreover, $E_0$ is non-zero if and only if $E$ is non-zero.
	
	We say that a standard section $\Phi_s$ is factorisable if $\Phi_s = \otimes_{v \leq \infty} \Phi_{s,v}$.
	For a factorisable standard section $\otimes_v \Phi_{s,v}$, we have the decomposition
	\[
	M_s \Phi_s = \prod_v M_{s,v} \Phi_{s,v}.
	\]
	For a non-archimedean place $v$ and a local standard section $\Phi_{s,v}$ of $I_v(\mu_v,s)$, we say $\Phi_{s,v}$ is unramified if $\Phi_{s,v}$ is right $K_v$-invariant.
	Then, by \cite[Proposition 9.3.2]{casselman}, for an unramified character $\mu_v$ and an unramified section $\Phi_{s,v}$, we have
	\[
	M_{s,v} \Phi_{s,v}(1) = \frac{L_v(s,\mu_v)}{L_v(s+1,\mu_v)} \Phi_{s,v}(1).
	\]
	Here $L_v(s,\mu_v) = (1-\mu_v(\varpi_v)|\varpi_v|_v^{-s})^{-1}$ is the local Euler factor when $\mu_v$ is unramified and is $1$ if $\mu_v$ is ramified.
	For a finite set of places $S$ containing all archimedean places, put
	\[
	L^S(s,\mu) = \prod_{v \not\in S} L_{v}(s,\mu_v).
	\]
	Set
	\[
	M_{s,v}^* = \frac{1}{L_v(s,\mu_v)}M_{s,v}.
	\]
	Kudla and Rallis proved the following statement in \cite[Proposition 4.3]{1988_Kudla-Rallis} and \cite[Proposition 5.5]{1992_Kudla-Rallis}:
	
	\begin{prop}\label{normalization}
	Let $v$ be a non-archimedean place.
	With the above notation, we have the following:
	\begin{enumerate}
	\item The normalized intertwining operator $M^*_{s,v}$ is non-zero and entire, i.e., for any $s$, for any $s$, there exists a function $\Phi_s \in I_v(\mu_v,s)$ such that $M_{s,v} \Phi_s$ is non-zero and for any standard section $\Phi_s$, $M^*_{s,v} \Phi_{s}$ is entire.
	\item Suppose $\mu_v = \chi_{V_v}$ for some two-dimensional quadratic space $V_v$.
	Then at $s=0$, we have
	\[
	M_{0,v}(R(V_v)) = R(V_v).
	\]
	\end{enumerate}
	\end{prop}

	For an archimedean place $v$, an integer $\ell$ and $\mu_v=\mathrm{sgn}^{\ell}$, let $\Phi_\infty^{\ul{\ell}}$ be a non-zero standard section of weight $\ul{\ell}=(\ell,\ldots,\ell)\in\bbZ^d$ in $I_\infty (\mu_\infty , s) = \bigotimes_{v \in \bfa} I_v(\mu_v,s)$.
	The following Lemma is well-known.
	For the proof, see \cite[Theorem 2.4]{Kudla-Yang}.
	
	\begin{lem}\label{const_term_Eis}
	Let $\Phi_{s,\fini} \otimes \Phi_\infty^{\underline{\ell}} = \otimes_{v < \infty} \Phi_{s,v} \otimes \Phi_\infty^{\ul{\ell}}$ be a factorisable standard section for $\ell \in \bbZ$.
	Let $S$ be a finite set of places such that $\Phi_{s,v}$ is unramified for $v \not \in S$.
	Then, we have
	\[
	E_0(1,s,\Phi_s^{\ul{\ell}}) = \Phi_s^{\ul{\ell}}(1) + M_{s}\Phi_s^{\ul{\ell}}(1)
	\]
	\[
	=\Phi_s^{\ul{\ell}} (1) + \left(    \frac{\pi (- \sqrt{-1})^\ell 2^{1-s}\Gamma(s)}{\Gamma(\alpha)\Gamma(\beta)} \right)^d \frac{L^S(s,\mu)}{L^S(s+1,\mu)} \left(\prod _{v \in S \setminus \{\infty\}}M_{s,v} \Phi_{s,v}(1) \right) \left(\prod_{v \not\in S} \Phi_{s,v}(1)\right)\Phi_{\infty,-s}^{\ul{\ell}}(1) ,
	\]
	where $d=[F:\bbQ]$, $\alpha=(s+1+\ell)/2$ and $\beta = (s+1-\ell)/2$.
	\end{lem}
	
	Put
	\[
	\xi(s, \ell) = \left(    \frac{\pi (- \sqrt{-1})^\ell 2^{1-s}\Gamma(s)}{\Gamma(\alpha)\Gamma(\beta)} \right)^d
	\]
	where $\alpha=(s+1+\ell)/2$ and $\beta = (s+1-\ell)/2$.
	If the induced representation $I_\infty(\mu_\infty,s)$ has a non-zero $\frakp_-$-finite vector, one of the following conditions hold:
	\begin{itemize}
	\item $s \in 2 \bbZ_{\geq0}$ and $\mu_v = \mathrm{sgn}$ for any $v \in \bfa$.
	\item $s \in -1+2\bbZ_{\geq 0}$ and $\mu_v = \mathbf{1}_v$ for any $v \in \bfa$.
	\end{itemize}
	
	Take a standard factorisable section $\Phi_s$ of $I(\mu,s)$.
	First, we consider the case $s \in 2 \bbZ_{\geq0}$ and $\mu_v = \mathrm{sgn}$ for any $v \in \bfa$.
	In this case, we may assume a standard factorisable section $\Phi_s$ can be expressed as $\Phi_s = \Phi_{s,\fini} \otimes \Phi_{s,\infty}^{\ul{k}}$.
	Put $k=s+1$.
	Then, $k \geq 3$ or $k=1$.
	In this case, $\Phi_{s,\infty}^{\ul{k}}$ lies in $L(\ul{k}) \subset \bigotimes_{v \in \bfa}I(\mu_v,k-1)$.
	
	Clearly, we have the following statement:
	
	\begin{lem}\label{k>3_odd}
	Suppose $k \in \bbZ_{\geq 3}$ and $k$ is odd.
	For any factorisable standard section $\Phi^{\ul{k}}_s = \Phi_{s,\fini} \otimes \Phi_{s,\infty}^{\ul{k}}$, we have
	\[
	E_0 (g,s,\Phi^{\ul{k}}_s)|_{s=k-1} = \Phi^{\ul{k}}_{k-1}(g), \qquad g \in G(\bbA).
	\]
	\end{lem}
	\begin{proof}
	The factor $\xi(s, k)$ is zero at $s=k-1$ and the other factor
	\[
	\frac{L^S(s,\mu)}{L^S(s+1,\mu)} \left(\prod _{v \in S \setminus \{\infty\}}M_{s,v} \Phi_{s,v}(1) \right) \Phi_{\infty,-s}^{\ul{k}}(1)
	\]
	is finite.
	Hence, by Lemma \ref{const_term_Eis}, we have
	\[
	E_0 (1,s,\Phi^{\ul{k}}_s)|_{s=k-1} = \Phi^{\ul{k}}_{k-1}(1).
	\]
	By the Iwasawa decomposition $P_0(\bbA_\bbQ)K=G(\bbA_\bbQ)$, we have $E_0 (g,s,\Phi^{\ul{k}}_s)|_{s=k-1} = \Phi^{\ul{k}}_{k-1}(g)$ for any $g \in G(\bbA)$.
	This completes the proof.
	\end{proof}
	
	Put $k=1$.
	Suppose $\mu^2=\mathbf{1}$ with $\mu_v = \mathrm{sgn}$ for any $v \in \bfa$.
	This case corresponds to Theorem \ref{KR}.
	
	\begin{lem}\label{realization}
	Let $V$ be a two-dimensional quadratic space over $F$ and $\Phi_s$ be a standard section in $I(s,\chi_V)$ such that $\Phi_0 \in R(V)$.
	The constant term $E(g,0,\Phi)$ is equal to $\Phi_0$.
	Hence the Eisenstein series gives the realization
	\[
	R(V) \longrightarrow \calA(G) \colon \Phi_0 \longmapsto E (\,\cdot\,,1,\Phi).
	\]
	\end{lem}
	\begin{proof}
	By \cite[Proposition 7.1]{1988_Kudla-Rallis}, we have
	\[
	M_{0} \Phi_0 = \Phi_0.
	\]
	This shows that the constant term of the Eisenstein series $E(g,s,\Phi)$ is equal to
	\[
	\Phi_0 + M_{w,0}\Phi_0 = 2 \Phi_0
	\]
	at $s=0$.
	This completes the proof.
	\end{proof}
	
	\begin{rem}\label{eigensp_int_op}
	For a two-dimensional quadratic space $V$, the composition of the intertwining operator $M_{0} \circ M_0$ acts on $I(0,\chi_V)$ as the identity map by \cite[Theorem 1 (iii)]{Arthur}.
	Hence the induced representation $I(0,\chi_V)$ decomposes as a direct sum of $+1$-eigenspace $I(0,\chi_V)^+$ and $-1$-eigenspace $I(0,\chi_V)^-$.
	By Theorem \ref{KR}, the space $R(\calC)$ is contained in $I(0,\chi_V)^-$ for an incoherent family $\calC$.
	Indeed, if $\Phi \in R(\calC)$ lies in $I(0,\chi_V)^+$, the Eisenstein series $E(\, \cdot \, , s, \Phi)$ is non-zero at $s=0$.
	Hence we have the realization $R(\calC) \longmapsto \calA(G)$.
	This contradicts to Theorem \ref{KR}.
	Conversely, the space $R(V)$ is contained in $I(0,\chi_V)^+$.
	Hence we have
	\[
	I(0,\chi_V)^+ = \bigoplus_W R(W), \qquad I(0,\chi_V)^- = \bigoplus_\calC R(\calC),
	\]
	where $W$ runs through all two-dimensional quadratic spaces over $F$ such that $\chi_W = \chi_V$ and $\calC = \{W_v\}$ runs through all incoherent families such that $\chi_{W_v} = \chi_V$.
	This is a key fact to show the Siegel-Weil formula.
	Moreover, the similar statement holds for local induced representations.
	For a place $v$ of $F$, let $V^\varepsilon$ be a two-dimensional quadratic space over $F_v$ such that $\varepsilon(V^\varepsilon) = \varepsilon$.
	If $v$ is non-archimedean, we can show that the representation $R(V^{\varepsilon})$ is the $\varepsilon \gamma(0,\chi_{V^\varepsilon},\psi_v)$-eigenspace of the local intertwining operator $M_{w,v,0}$ by \cite[Lemma 3.1]{ikeda}.
	Here $\gamma(s,\chi_{V^\varepsilon},\psi_v)$ is equal to
	\[
	\gamma(s,\chi_{V^\varepsilon},\psi_v) = \varepsilon(s,\chi_{V^\varepsilon},\psi_v) \frac{L(1-s,\chi_{V^\varepsilon}^{-1})}{L(s,\chi_{V^\varepsilon})}
	\]
	and $\varepsilon(s,\chi_{V^\varepsilon},\psi_v)$ is the local epsilon factor associated to $\psi_v$.
	Note that the eigenvalue $\varepsilon \gamma(0,\chi_{V_v},\psi_v)$ is depending on the definition of the Weil representation associated to $\psi$.
	Of course, the same statement for an archimedean case holds by the straightforward computations.
	\end{rem}
	
	For a non-quadratic Hecke character $\mu$, by the functional equation, we have the following:
	
	\begin{lem}\label{k=1_non-quad}
	Suppose $\mu^2 \neq \mathbf{1}$.
	Let $\Phi_s$ be a standard section of $I(\mu,s)$.
	Then we have
	\[
	E(g,0,\Phi) = E(g,0,M_s \Phi), \qquad g \in G(\bbA_F).
	\]
	\end{lem}
	
	Next, we assume $s \in -1+2\bbZ_{\geq 0}$ and $\mu_v = \mathbf{1}_v$ for any $v \in \bfa$.
	The case of $k = s+1 \geq 4$ is similar to Lemma \ref{k>3_odd}.
	
	\begin{lem}\label{k>3_even}
	Suppose $k \in \bbZ_{\geq 3}$ and $k$ is even.
	For any factorisable standard section $\Phi^{\ul{k}}_s = \Phi_{s,\fini} \otimes \Phi_{s,\infty}^{\ul{k}}$, we have
	\[
	E_0 (g,s,\Phi^{\ul{k}}_s)|_{s=k-1} = \Phi^{\ul{k}}_{k-1}(g), \qquad g \in G(\bbA).
	\]
	\end{lem}
	
	We assume $s=1$.
	In this case, $k$ is equal to $2$.
	
	\begin{lem}\label{k=2}
	For any factorisable standard section $\Phi_s^{\ul{2}} = \Phi_{s,\fini} \otimes \Phi_{s,\infty}^{\ul{2}}$, we have
	\[
	E_0 (1,s,\Phi^{\ul{2}})|_{s=1} =\Phi_{1}^{\ul{2}}(1) - \frac{3}{\pi} \Phi_{-1}^{\ul{2}}(1)  
	\]
	if $F = \bbQ$, $\mu=\mathbf{1}$ and $S=\{\infty\}$.
	We also have
	\[
	E_0 (1,s,\Phi^{\ul{2}})|_{s=1} =\Phi_{1}^{\ul{2}}(1)
	\]
	if one of the following condition holds:
	\begin{itemize}
	\item $F \neq \bbQ$.
	\item $\Phi_{1,v} \in \mathrm{St}_v$ for some non-archimedean place $v$.
	\item $\mu \neq \mathbf{1}$.
	\end{itemize}
	\end{lem}
	
	\begin{proof}
	The factor
	\[
	\frac{L^S(s,\mu)}{L^S(s+1,\mu)} \left(\prod _{v \in S \setminus \{\infty\}}M_{s,v} \Phi_{s,v}(1) \right) \Phi_{\infty,-s}^{\ul{2}}(1)
	\]
	has at most simple pole at $s=1$.
	Suppose $F \neq \bbQ$.
	In this case, the factor $\xi(s,2)$ has a zero of order $d > 1$ at $s=1$.
	Hence we have
	\[
	E_0 (1,s,\Phi_s^{\ul{2}})|_{s=1} =\Phi_{1}^{\ul{2}}(1).
	\]
	Suppose $F = \bbQ$, $\mu=\mathbf{1}$ and $S=\{\infty\}$.
	Then ${L^S(s,\mu)}/{L^S(s+1,\mu)}$ has a simple pole at $s=1$ and $\pi (-\sqrt{-1})^2 2^{1-s}\Gamma(s)/\Gamma(\alpha)\Gamma(\beta)$ has a simple zero at $s=1$.
	By computing the residue, we have
	\[
	E_0 (1,s,\Phi_s^{\ul{2}})|_{s=1} =\Phi_{1}^{\ul{2}}(1) - \frac{3}{\pi} \Phi^{\ul{2}}_{-1}(1).
	\] 
	Suppose $F=\bbQ$ and $\Phi_{s,v}$ belongs to the Steinberg representation $\mathrm{St}_v$ at $s=1$.
	For such a place $v$, we have $M_{w,s,v}\Phi_{s,v}=0$ at $s=1$ by Lemma \ref{local}.
	Hence the lemma follows.
	Finally suppose $\mu \neq 1$.
	In this case, since the factor $\xi(s,2)$ has a simple zero at $s=1$ and the $L$-function $L^S (s,\mu)/L^S (s+1,\mu)$ does not have a pole at $s=1$, the assertion is clear.
	This completes the proof.
	\end{proof}
	
	
	\subsection{Siegel-Weil formula}
	Before the proof of Theorem \ref{main}, we refer to the Siegel-Weil formula.
	Let $V$ be a quadratic space over $F$ with dimension $m$.
	Suppose $m$ is even.
	We then obtain the Weil representation $\omega_\psi$ on the Schwartz-Bruhat space $\calS(V(\bbA))$ where $\psi$ is the fixed additive character of $F \bs \bbA_F$ (cf.~\S \ref{ind_rep_theta_corr}).
	Then we have a map
	\[
	\omega_\psi \longrightarrow I(\chi_V, m/2 - 1) \colon \varphi \longmapsto (g \longmapsto \omega_\psi(g)\varphi(0)).
	\]
	For $\varphi \in \calS(V)$, let $\Phi_{\varphi,s}$ be a standard section of $I(\chi_V,m/2-1)$ such that
	\[
	\Phi_{\varphi,m/2-1} (g) = \omega_\psi(g)\varphi(0).
	\]
	We get the Eisenstein series $E(\, \cdot \, , s, \Phi_\varphi)$.
	
	Set $H = \mathrm{O}(V)$.
	Put
	\[
	\theta(g,h;\varphi) = \sum_{x \in V(F)} \omega_\psi(g)\varphi(h^{-1} x), \qquad \varphi \in \omega_\psi, \, g \in \SL_2(\bbA_F), \, h \in H(\bbA_F).
	\]
	The Siegel-Weil formula for anisotropic case states the relation between the Eisenstein series $E(\, \cdot \, , s, \Phi_\varphi)$ and the theta function $\theta(g,h;\varphi)$:
	
	\begin{thm}[\cite{1994_Kudla-Rallis}]
	Suppose $V$ is anisotropic.
	Then for any $\varphi \in \omega_\psi$, we have
	\[
	E(g,m/2-1,\Phi_\varphi) = \kappa \int_{H(F) \bs H(\bbA)} \theta(g,h;\varphi) \, dh, \qquad g \in \SL_2(\bbA_F).
	\]
	Here 
	\[
	\kappa = 
	\begin{cases}
	1 &\text{if $m > 2$}\\
	2 &\text{if $m=2$}.
	\end{cases}
	\]
	\end{thm}
	
	
	For an archimedean place $v$, the image of the local Weil representation $\omega_{\psi_v}$ is the (limit of) discrete series representation of $\SL_2(\bbR)$ of weight $m/2$ if $V_v$ is positive definite.
	Suppose $V_v$ is positive definite for any $v \in \bfa$.
	Then $V$ is anisotropic.
	Hence the certain space of nearly holomorphic Eisenstein series is isomorphic to the space of theta integrals.
	
	\subsection{Proof of Theorem \ref{main}}
	By Lemma \ref{int_par}, we have (1).
	We first treat the case $k \geq 3$.
	By Lemma \ref{local}, we may assume that $s = k-1$ and $\mu_v=\mathrm{sgn}^k$ for any $v \in \bfa$.
	If $k \in \bbZ_{\geq 1}$, the space of $\frakp_-$-finite vectors in $\mathrm{Ind}_{B(\bbA)}^{G(\bbA)}(\mu |\cdot|^{k-1})$ is of the form
	\[
	I_\fini (\mu,k-1) \otimes L(\ul{k}).
	\]
	This is irreducible by the irreducibility of local induced representations.
	Hence, we obtain a map
	\begin{align}\label{k=3_BG}
	\calN(G,\chi_{\ul{k}}) \longrightarrow \bigoplus_{\mu \in \mathfrak{X}_{(-1)^k}} \left(I_\fini (\mu_v,k-1) \otimes L({\ul{k}}) \right) \colon \varphi \longmapsto \varphi_0.
	\end{align}
	
	\begin{proof}[Proof of Theorem \ref{main} (2)]
	By (\ref{k=3_BG}), we have an injective map
	\[
	\calE\calN(G,\chi_{\ul{k}}) \xhookrightarrow{\quad} \bigoplus_{\mu \in \mathfrak{X}_{(-1)^k}} \left(I_\fini (\mu_v,k-1) \otimes L({\ul{k}}) \right).
	\]
	Take a function $\Phi$ in the right hand side.
	Let $\Phi_s$ be a standard section such that $\Phi_{k-1} = \Phi$.
	Then by Lemma \ref{k>3_odd} and Lemma \ref{k>3_even}, for the standard section $\Phi_s$, the constant term of $E(\,\cdot\,,k-1,\Phi)$ is equal to $\Phi_{k-1}$.
	Hence we obtain the inverse map.
	This completes the proof.
	\end{proof}
	
	Next we treat the case of $k=1$.

	\begin{proof}[Proof of Theorem \ref{main} (3)]
	Fix a congruence subgroup $\Gamma$ of $\SL_2(F)$.
	We denote by $\calE_1(\Gamma)$ the orthogonal complement of $S_{\ul{1}}(\Gamma)$ in $M_{\ul{1}}(\Gamma)$.
	By (\ref{AF_vs_MF}), the space $\calE_1(\Gamma)$ corresponds to the space of holomorphic $K_\Gamma$-fixed automorphic forms on $G(\bbA)$ of weight $1$.
	Here $K_\Gamma$ is the closure of $\Gamma$ in $G(\bbA_\fini)$.
	
	For a Hecke character $\mu \in \mathfrak{X}_{-1}$, let $\calE(0)$ be the $\bbC$-vector space spanned by the Eisenstein series $E(\, \cdot \, ,0,\Phi)$ for a standard section $\Phi_s$ of $I(s,\mu)$.
	Note that such the Eisenstein series $E(g,s,\Phi)$ is finite at $s=0$ by Lemma \ref{normalization}.
	Let $\calE(0,K_{\Gamma})$ be the subspace of $\calE(0)$ consisting of $K_\Gamma$-fixed Eisenstein series.
	By \cite[Theorem 8.3]{85_shimura}, we have the isomorphism $\calE(0,K_\Gamma) \cong \calE_1(\Gamma)$ as a $\bbC$-vector space by (\ref{AF_vs_MF}).
	We obtain the correspondence
	\[
	\bigcup_\Gamma \calE_1(\Gamma) \cong \{ \varphi \in \calE\calN(G,\chi_1) \mid Y_v \cdot \varphi = 0 \; \text{for all $v \in  \bfa$}\}
	\]
	by the isomorphism (\ref{AF_vs_MF}), where $\Gamma$ runs through all congruence subgroups of $\SL_2(F)$ and $Y_v$ is a root vector of weight $-2$ in $\mathfrak{sl}_2(F_v) \otimes \bbC$.
	This shows that any automorphic form in $\calE\calN(G,\chi_1)$ is a finite sum of Eisenstein series.
	Since any automorphic form in $\calE\calN(G,\chi_{\bf{1}})$ generates a highest weight module as a $(\frakg,K_\infty)$-module, we then have the isomorphism
	\[
	\calE\calN(G,\chi_{\ul{1}}) \cong\bigoplus_{\mu \in \mathfrak{X}_{-1}/\sim} \{(\Phi_s + M_s\Phi_s)|_{s=0} \mid \Phi \in I_\fini(\mu,0) \otimes D_{\ul{1}}\} \oplus \bigoplus_{V}R(V)
	\]
	as a $G(\bbA_\fini) \times (\frakg,K_\infty)$-module by the constant term, by Lemma \ref{KR} and Lemma \ref{k=1_non-quad}.
	For a non-quadratic character $\mu$, the space $\{\Phi_0 + M_0 \Phi_0 \mid \Phi_0 \in I_\fini(\mu,0)\otimes D_{\ul{1}}\}$ is isomorphic to $I(\mu,0)$ as $G(\bbA_\fini)\times(\frakg,K_\infty)$-module.
	Indeed, the map $\Phi \longmapsto \Phi + M_0\Phi$ is injective and intertwining for $\Phi \in I_\fini(\mu,0) \otimes D_{\ul{1}}$ with a non-quadratic Hecke character $\mu$.
	Note that by Theorem \ref{KR}, for any incoherent family $\calC$, the representation $R(\calC)$ does not occur in $\calE\calN(G)$.
	This completes the proof.
	\end{proof}
	
	Finally we proof the case of $k=2$.
	
	\begin{proof}[Proof of Theorem \ref{main} (4)]
	The proof for the case of $F \neq \bbQ$ is similar to the proof for Theorem \ref{main} (2).
	
	Suppose $F=\bbQ$.
	For a non-archimedean place $v$, put
	\[
	M^{(v)} = \bigotimes_{p < \infty, p \neq v} I_p(\mathbf{1}_p,1) \otimes \mathrm{St}_v \otimes L(2) \subset I(\mathbf{1},1).
	\]
	By Lemma \ref{k=2}, the representation $M^{(v)}$ occur in $\calE\calN(G,\chi_2)$.
	Similarly, the induced representation $I(\mu,1)$ and the trivial representation $\bbC$ occur in $\calE\calN(G,\chi_2)$ if $\mu \neq \mathbf{1}$.
	We show that the subrepresentation 
	\begin{align*}\label{sub_rep_k=2}
	\tau = \bbC \oplus \bigoplus_{v < \infty} M^{(v)} \oplus \bigoplus_{\mu \in \mathfrak{X}_1} I(\mu,1) 
	\end{align*}
	of $\calE\calN(G,\chi_2)$ is a maximal subrepresentation of it.
	By Proposition \ref{const_term_NHAF} and Lemma \ref{local} (1), $\calE\calN(G,\chi_2)$ can be identified with a subrepresentation of the subspace of all $\frakp_-$-finite vectors in $\bigoplus_{\mu \in \mathfrak{X}_1}(I(\mu,1) \oplus I(\mu,-1))$.
	Such a space of all $\frakp_-$-finite vectors is equal to the following space:
	\[
	\bigoplus_{\mu \in \mathfrak{X}_1} \left(I_\fini(\mu,1) \otimes L(2) \oplus I_\fini (\mu,-1) \otimes N(0)^\vee\right).
	\]
	Here $N(0)^\vee$ is the contragredient representation of the Verma module $N(0)$ (see \S \ref{BGG_cat}).
	Take a nearly holomorphic automorphic form $\varphi \in \calE\calN(G,\chi_2)$ of weight $2$.
	Then there exist $f \in \bigoplus_{\mu \in \mathfrak{X}_1} I_\fini(\mu,1) \otimes L(2) $ and $g \in \bigoplus_{\mu \in \mathfrak{X}_1} I_\fini (\mu,-1) \otimes N(0)^\vee$ such that $\varphi _0 = f + g$.
	For a root vector $Y \in \mathfrak{sl}_2(\bbC)$ of weight $-2$, we have $Y \cdot f = 0$.
	Then $Y \cdot \varphi$ is of weight $0$ and hence $Y \cdot \varphi$ is a constant function.
	Since constant functions generates the trivial representation, $(Y \cdot \varphi)_0 = Y \cdot \varphi_0$ belongs to $I(\mathbf{1},-1)$.
	This shows that $g$ lies in $I_\fini(\mathbf{1},-1) \otimes N(0)^\vee$.
	Summarizing that, the representation $\calE\calN(G,\chi_2)$ may be identified with a subrepresentation of 
	\[
	\Pi = I_\fini (\mathbf{1},-1) \otimes N(0)^\vee \oplus \bigoplus_{\mu \in \mathfrak{X}_1} \left(I_\fini(\mu,1) \otimes L(2) \right).
	\]
	We now obtain the filtration
	\[
	\tau \subset \calE\calN(G,\chi_2)  \subset \Pi.
	\]
	Then the quotient space $\Pi / \tau$ is isomorphic to
	\begin{align}\label{pi/tau}
	\mathrm{triv}_{G(\bbA_\fini)} \otimes L(2) \oplus \mathrm{triv}_{G(\bbA_\fini)} \otimes L(2).
	\end{align}
	Since $\tau$ has no unramified vector of weight $2$ and $\calE\calN(G,\chi_2)$ has an unramified vector of weight $2$, the quotient $\calE\calN(G,\chi_2) /\tau$ is non-zero.
	Since the representation (\ref{pi/tau}) has length $2$, $\calE\calN(G,\chi_2)/\tau$ is isomorphic to the representation (\ref{pi/tau}) or an irreducible representation of the form $\mathrm{triv}_{G(\bbA_\fini)} \otimes L(2)$.
	Note that the dimension of unramified vectors in (\ref{pi/tau}) of weight $2$ is $2$.
	If $\calE\calN(G,\chi_2)/\tau$ is isomorphic to the representation (\ref{pi/tau}), this contradicts to $\dim N_2(\SL_2(\bbZ)) = 1$.
	Hence $\calE\calN(G,\chi_2)/ \tau$ is irreducible and is generated by $E_{2,\bbA}$.
	Since $\pi$ contains $\bbC$ and $M^{(v)}$ for any $v < \infty$, we have
	\[
	\calE\calN(G,\chi_2) \cong \left( \bigoplus_{\begin{smallmatrix}\mu \in \mathfrak{X}_1 \\ \mu \neq \mathbf{1} \end{smallmatrix}} I_\fini(\mu,1)\otimes L(2) \right) \oplus \pi. 
	\]
	This completes the proof.
	\end{proof}
	
	This completes the proof of Theorem \ref{main}.
	
	
	
	\subsection{proof of Corollary \ref{cor_main_2}}
	
	Suppose $F \neq \bbQ$.
	In this case, since any $(\frakg,K_\infty)$-submodule of $\calN(G)$ is highest weight representation, we have the equality $\calN(G) = \mathcal{H}(G)$.

	Suppose $F=\bbQ$.
	By Theorem \ref{main}, the following representations are generated by highest weight vectors:
	\[
	\calE\calN(G,\chi_k), \qquad \bigoplus_{\mu \in \mathfrak{X}_1, \mu \neq \mathbf{1}} I(\mu,1) , \qquad M^{(v)}, \qquad \bbC
	\]
	for $k \in \bbZ_{\geq 1}$ with $k \neq 2$ and all non-archimedean place $v$.
	By the proof of Theorem \ref{main} (4), the assertion follows.

	\appendix
	\section{The space of Eisenstein series}\label{sp_Eis_ser}
	In this appendix, we conjecture the relationship between Eisenstein series and nearly holomorphic automorphic forms.
	
	We keep the notation of \S \ref{def_NHMF}.
	Fix a minimal parabolic $\bbQ$-subgroup $P_0$ of $G$ and a Levi decomposition $P_0=M_0N_0$.
	We define the Weyl group $W=W_G$ by $\mathrm{Norm}_G(M_0) / M_0$.
	Let $P=MN$ be a standard parabolic subgroup of $G$ with $M$ the standard Levi subgroup.	
	We regard $\fraka_M^* = \Hom(A_M^\infty,\bbC^\times) \cong \Hom(M(\bbA)/M^1(\bbA),\bbC^\times)$ as the set of characters of $M(\bbA)$.
	Let $(\fraka_M^G)^* = \Hom(A_M^\infty / A_G^\infty,\bbC^\times)$.
	We regard $\lambda \in (\fraka_M^G)^*$ as a character of $M(\bbA)$.
	The modulus character $\rho_P$ of $P(\bbA)$ can be regarded as an element of $(\fraka_M^G)^*$.


	For a standard Levi subgroup $M$, we set
\[
	W(M) = \left\{ 
	w \in W \,\middle| \,
	\begin{array}{ll}
	\text{$\bullet$ $wMw^{-1}$ is a standard Levi subgroup of $G$} 
	\\
	\text{$\bullet$ $w$ has a minimal length in $wW_M$}
	\end{array}
	\right\}.
	\]
	Recall that two standard parabolic subgroups $P=MN$ and $P'=M'N'$ are associated if there exists $w\in W(M)$ such that $M'=wMw^{-1}$.
	If $P'$ and $P$ are associated we write $P\sim P'$.
	Take $w\in W(M)$.
	Put $M'=wMw^{-1}$ and let $P'=M'N'$ be the standard parabolic subgroup with Levi subgroup $M'$.

	Let $\tau$ be an irreducible cuspidal automorphic representation of $M(\bbA)$.
	We tacitly assume that the central character $\chi_\tau$ is trivial on $A_G^\infty$.
	Recall that a cuspidal datum is a pair $(M, \tau)$ such that $M$ is a Levi subgroup of $G$ and that $\tau$ is an irreducible cuspidal automorphic representation.
	Let $w\in W(M)$. 
	The irreducible cuspidal automorphic representation $\tau^w$ of $M'(\bbA)$ is defined by $\tau^w(m')=\tau(w^{-1}m'w)$ for $m'\in M'(\bbA)$.
	Similarly, for $\lambda\in (\fraka_M^G)^*$, $\lambda^w\in (\fraka_M^G)^*\longmapsto (\fraka_{M'}^G)^*$ is defined by $\lambda^w(m')=\lambda(w^{-1}m'w)$ for $m\in M'(\bbA)$.
	A cuspidal datum $(M,\tau)$ is called regular if we have $\tau^w \not \cong \tau$ for any $1\neq w \in W$ with $wMw^{-1} =M$.
	Two cuspidal data $(M, \tau)$ and $(M', \tau')$ are called equivalent if there exists $w\in W(M)$ such that $M'=wMw^{-1}$ and that $\tau'=\tau^w$.

	Let $\mathcal{A}(G)_{(M, \tau)}$ is the subspace of automorphic forms in $\mathcal{A}(A \bs G)$ with cuspidal support $(M, \tau)$.
	The space $\mathcal{A}(A \bs G)$ is decomposed as
\[
\mathcal{A}(A \bs G)=\bigoplus_{(M, \tau)} \mathcal{A}(A \bs G)_{(M, \tau)}.
\]
Here, $(M, \tau)$ runs through all equivalence classes of cuspidal datum. 
(See \cite{MW}, III.2.6, \cite{FS} Theorem 1.4).
Put $\calN(A \bs G)_{(M, \tau)}=\calN(A \bs G)\cap \mathcal{A}(A \bs G)_{(M, \tau)}$.
	\begin{quest}
	When $\calN(A \bs G)_{(M, \tau)}\neq 0$ ?
	\end{quest}

	Let $\calA_{\mathrm{cusp}}(A \bs G)$ be the space of cusp forms on $G(\bbQ) \bs G(\bbA)$.
	The $\tau$-isotypic subspace of the space $\mathcal{A}_{\mathrm{cusp}}(M(\bbA))$ of cusp form on $M(\bbQ)\backslash M(\bbA)$ is denoted by $\mathcal{A}_{\mathrm{cusp}}(M(\bbA))_\tau$.
	Let $I_P^G(\tau)$ be the space which consists of all smooth functions $f : M(\bbQ) N(\bbA) \backslash G(\bbA) \longrightarrow \bbC$ such that the following conditions (1) and (2) hold:
	\begin{itemize}
	\item[(1)] $f$ is right $K$-finite.
	\item[(2)] for any $k \in K$, a function on $M$ defined by $m \longmapsto m^{-\rho_P}f(mk)$ belongs to $\mathcal{A}_{\mathrm{cusp}}(M(\bbA))_\tau$.
	\end{itemize}
	Put $\mathcal{I}(\tau)=\mathrm{Ind}_{P(\bbA)\cap K}^{K}(\mathcal{A}_{\mathrm{cusp}}(M(\bbA))_\tau)$.
	Then we have an isomorphism $I_P^G(\tau)\simeq \mathcal{I}(\tau)$.
	For $f \in I_P^G(\tau)$ and $\lambda \in (\fraka_M^G)^*$, put
	\[
	f_\lambda (mnk) = m^\lambda f(mk), \qquad m \in M(\bbA),\, n \in N(\bbA),\, k \in K.
	\]
	Then, we have $f_\lambda \in I_P^G(\tau\otimes\lambda)$.
	Then we have a $\bbC$-vector space  isomorphism $\mathcal{I}_P^G(\tau\otimes\lambda)\simeq \mathcal{I}(\tau)$ for any $\lambda$.
	Then, the product $\mathcal{I}(\tau)\times  (\fraka_M^G)^*$ can be considered as a fiber bundle over $ (\fraka_M^G)^*$ such that the fiber at $\lambda\in  (\fraka_M^G)^*$ is $I_P^G(\tau \otimes \lambda)$.
	We denote this fiber bundle by 
	\[
\mathcal{F}_{ (\fraka_M^G)^*}  \mathcal{I}(\tau)=
	\coprod_{\lambda \in (\fraka_M^G)^*} I_P^G(\tau \otimes \lambda).
	\]
	For an irreducible representation $\sigma$ of $K$, let  $\mathcal{I}(\tau)_\sigma$ be the space of functions $f \in \mathcal{I}(\tau)_\sigma$ with $K$-type $\sigma$.
	Then $\mathcal{I}(\tau)_\sigma$ is a finite-dimensional $\bbC$-vector space.
	The space $\mathcal{I}(\tau)$ is endowed with the topology induced from the algebraic direct sum
	\[
	\mathcal{I}(\tau) = \bigoplus_{\sigma} \mathcal{I}(\tau)_\sigma.
	\]
	The fiber bundle $\mathcal{F}_{ (\fraka_M^G)^*}  \mathcal{I}(\tau)$ is endowed with the topology induced from the product topology of  $\mathcal{I}(\tau)\times (\fraka_M^G)^*$.
	For $f \in I_P^G(\tau)$, the function of the form $\lambda \longmapsto f_\lambda$ is called a standard section of $\mathcal{F}_{ (\fraka_M^G)^*}  \mathcal{I}(\tau)$.
	

	Let $P=MN$ be a standard parabolic subgroup of $G$ and $(M,\tau)$ a cuspidal datum.
	For any $f \in \mathcal{I}_P^G(\tau)$, we define the Eisenstein series $E(g,\lambda,f)$ by
	\[
	E(g,\lambda,f) = \sum_ {\gamma \in P(\bbQ) \backslash G(\bbQ)} f_\lambda (\gamma g), \qquad g \in G(\bbA),\, \lambda \in (\fraka_M^G)^*.
	\]
	Due to Langlands' theory \cite{Langlands}, the Eisenstein series $E(g,\lambda,f)$ is absolutely convergent on some open set of $(\fraka_M^G)^*$ and it can be continued to a meromorphic function on the entire space $(\fraka_M^G)^*$.
	The intertwining operator 
	\[
	M(w,\tau\otimes \lambda) : I_P^G(\tau \otimes \lambda) \longrightarrow I_{P'}^G(\tau^w\otimes\lambda^w)
	\]
	is defined by the integral
	\[
	M(w,\tau\otimes\lambda)f(g) = \int_{N'(\bbA) \cap wN(\bbA)w^{-1} \backslash N'(\bbA)} f(w^{-1}ng) \, dn,
\qquad 
f\in \mathcal{I}_P^G(\tau\otimes\lambda)
	\]
	as long as the integral converges.
	It is well-known that $M(w,\tau\otimes\lambda)$ converges absolutely on some open subset of $(\fraka_M^G)^*$ and it can be continued to a meromorphic function on the entire space $(\fraka_M^G)^*$ (see \cite[Proposition II.1.6]{MW}).
	The set of singular points $\frakX_w$ of  $M(w,\tau\otimes\lambda)$ is contained in a locally finite union of hyperplanes in $(\fraka_M^G)^*$.
	If necessary, we enlarge $\frakX_w$ so that the intertwining operators $M(w,\tau\otimes\lambda)$ are isomorphism for any $\lambda \not \in \frakX_w$. 
	We may also assume that the cuspidal datum $(M,\tau\otimes\lambda)$ is regular for any $\lambda \in \frakX_w$.
	Note that the set of poles of the Eisenstein series is contained in $\frakX = \bigcup_{w\in W(M)} \frakX_w$.

	Let $P'=M'N'$ be a standard parabolic subgroup such that $P'\sim P$.
	Put 
\[
W(M,M') = \{w \in W(M) \mid \text{$wMw^{-1}=M'$}\}.
\]
	By \cite[Proposition II.1.7]{MW}, the constant term $E_{P'}$ of $E(\,\cdot\,,\lambda,f)$ along a standard parabolic subgroup $P'=M'N'$ is equal to
	\begin{align*}
	E_{P'}(g,\lambda,f) &= \int_{N'(\bbQ)\backslash N'(\bbA)} E(ng,\lambda,f) \, dn  \\
						&= \sum_{w\in W(M,M')} M(w,\tau\otimes\lambda) f(g),
	\end{align*}
	for $\lambda \not\in \frakX$.
	Suppose that $P''=M''N''\sim P'=M'N'\sim P=MN$.
	Then we have a functional equation
	\[
	M(w',\tau^w\otimes\lambda^w) \circ M(w,\tau\otimes\lambda) = M(w'w,\tau\otimes\lambda),
	\]
for $w\in W(M, M')$,  $w'\in W(M', M'')$.

	Now we consider the fiber bundle $\mathcal{F}_{ (\fraka_{wMw^{-1}}^G)^*}  \mathcal{I}(\tau^w)$ over $(\fraka_{wMw^{-1}}^G)^*$ as a fiber bundle over $(\fraka_M^G)^*$ by the pullback by $\lambda \longmapsto \lambda^w$.
	Then we take the fiber product 
\[
\mathcal{F}_{ (\fraka_M^G)^*}  
\Bigl(\,
\prod_{w\in W(M)}\mathcal{I}(\tau^w)\,\Bigr)
\]
of 
\[
\left\{ \mathcal{F}_{ (\fraka_{wMw^{-1}}^G)^*}  \mathcal{I}(\tau^w)\, 
|\, w\in W(M)\right\}
\]
over $(\fraka_M^G)^*$.
	Let $f_\lambda$ be a standard section of $\mathcal{F}_{(\fraka_M^G)^\ast}\mathcal{I}(\tau)$.
	Then $\lambda \longmapsto M(w, \tau\otimes\lambda)f_\lambda$ can be considered as a meromorphic section of $\mathcal{F}_{ (\fraka_{wMw^{-1}}^G)^*}  \mathcal{I}(\tau^w)$ defined for $\lambda\notin\frakX$.
	Thus the collection $(M(w, \tau\otimes \lambda)f_\lambda)_{w\in W(M)}$ can be considered as a meromorphic section of the fiber product bundle $\mathcal{F}_{ (\fraka_M^G)^*}  \Bigl(\,\prod_{w\in W(M)}\mathcal{I}(\tau^w)\,\Bigr)$.

	Let $\calV(M,\tau)$ be the closure of the loci of these sections, where $f_\lambda$ runs through all standard sections.
	Put
	\[ 
	\calE(M,\tau) = \calV(M,\tau) \cap \bigoplus_{\begin{smallmatrix}{P'=M'N'}\\{P'\sim P}\end{smallmatrix}} \left(\bigoplus_{w\in W(M,M')} \left(\mathcal{I}_{P'}^{G}(\tau^w) \right) \right).
	\]
	Here, 
	\[
	\bigoplus_{\begin{smallmatrix}{P'=M'N'}\\{P'\sim P}\end{smallmatrix}}  \left(\bigoplus_{w\in W(M,M')} \left(\mathcal{I}_{P'}^{G}(\tau^w) \right) \right)
	\] 
	is considered as the fiber of the fiber bundle $\mathcal{F}_{ (\fraka_M^G)^*}  \Bigl(\,\prod_{w\in W(M)}\mathcal{I}(\tau^w)\,\Bigr)$ at $\lambda=0$.
	If $0 \not \in \frakX$, we have the isomorphism $\mathcal{I}_P^G(\tau) \cong \calE(M,\tau)$, since in this case $M(w, \tau)$ is an isomorphism for any $w\in W(M)$.
	
	Suppose $(f_w)_w \in \calE(M,\tau)$.
	Then, there exists a sequence 
\[
(f_{w,i})_w \in \calE(M,\tau\otimes\lambda_i^w),
\qquad i=1,2, \ldots
\] with $\lambda_i \in (\fraka_M^G)^*\setminus\frakX$ such that $\lim_{i \rightarrow \infty} f_{w,i} = f_w$ for all $w$.
	For such a sequence $(f_{w,i})_w$, the limit
	\[
	\lim_{i \rightarrow \infty} E(g ,\lambda_i,f_{1,i})
	\]
	exists, since the limit of any constant term exists.
	This limit does not depend on the choice of the sequence $(f_{w,i})_w $ and will be denoted by $E(g, (f_w)_w)$.
	Let $\calE_0(M,\tau)$ be a $\bbC$-vector space which is spanned by Eisenstein series $E(g, (f_w)_w)$ for $(f_w)_w \in \calE(M,\tau)$.
	Then we have $\calE_0(M,\tau)\subset \mathcal{A}(A \bs G)_{(M, \tau)}$.
Note that if $\tau$ is regular, then the map $\calE(M,\tau) \longrightarrow \calE_0(M, \tau)$ is an isomorphism.
	We then state the conjecture as follows:
	
	
	\begin{conj}\label{conj}
\begin{itemize}
\item[(1)] $\calN(A \bs G)_{(M, \tau)}\subset \calE_0(M, \tau)$.
\item[(2)] The action of $\mathcal{Z}$ on $\calN(A \bs G)$ is semisimple.
\item[(3)] If $\calN(A \bs G)_{(M, \tau)}\neq 0$, then the infinitesimal character of $\tau$ is integral.
\end{itemize}
	\end{conj}
Note that (1) implies (2), since the action of $\mathcal{Z}$ on $\calE_0(M, \tau)$ is semisimple.
	\begin{rem}
	Let $\lambda \longmapsto f_\lambda$ be a holomorphic section of $I_P^G(\tau\otimes \lambda)$.
	Then there exists a polynomial $Q(\lambda)$ of $\lambda\in (\fraka_M^G)^\ast$ such that $Q(\lambda)E(g, \lambda, f_\lambda)$ is holomorphic at $\lambda=0$.
	In \cite{FS}, it is proved that the space $\mathcal{A}(G(\bbQ)A_G^\infty\backslash G(\bbA))_{(M, \tau)}$ is generated by the derivatives of $Q(\lambda)E(g, \lambda, f_\lambda)$ as $f_\lambda$ runs through all holomorphic sections.
	The space $\calE_0(M, \tau)$ can be considered as the space of the leading terms of $Q(\lambda)E(g, \lambda, f_\lambda)$.
	\end{rem}
	

\end{document}